\documentclass[reqno, 12pt, a4paper]{amsart}
\usepackage[dvipsnames]{xcolor}
\usepackage{tikz-cd}
\usepackage[utf8]{inputenc}
\usepackage[T1]{fontenc}
\usepackage{amsmath}
\usepackage{amsfonts,amssymb}
\usetikzlibrary{matrix}
\usepackage{tensor}
\usepackage{bbm}

\usepackage{algorithm2e}
\graphicspath{{./zdjecia_uniw_melonow/}}

\newcommand{\tenQ}{ \raisebox{-.318\height}{  \includegraphics[height=3.399ex]{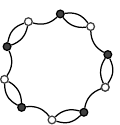}}}

\newcommand{\Aone}{\raisebox{-.30\height}{\hspace{-2pt}\includegraphics[height=3.05ex]{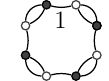}}}

\newcommand{\E}{\mathbb{E}}
\newcommand{\Qone}{\raisebox{-.3\height}{\hspace{1pt}\includegraphics[height=2.4ex]{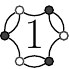}}}

\newcommand{\vone}{\raisebox{-.3\height}{\includegraphics[height=2.3ex]{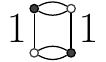}}}

\makeatletter
\newcommand{\vast}{\bBigg@{4.030}}
\newcommand{\Vast}{\bBigg@{12.30}}
\makeatother

\usepackage{tabularx}
\newcolumntype{L}{>{\arraybackslash}m{12cm}}

\usepackage[cal=euler, scr=boondoxo]{mathalfa}

\def\[#1\]{%
  \begin{align}#1%
  \end{align}%
}

\usepackage{mathabx}

\definecolor{azulESI}{HTML}{1266AE}

\definecolor{AZULESI}{HTML}{1266AE}
\usepackage{enumerate}

\usepackage{nicematrix}
\usepackage{soul}
\usepackage[utf8]{inputenc}

\makeatletter
\ifcase \@ptsize \relax
  \newcommand{\nano}{\@setfontsize\miniscule{3.5}{4.5}}
\or
  \newcommand{\nano}{\@setfontsize\miniscule{4.5}{5.5}}%
\or
  \newcommand{\nano}{\@setfontsize\miniscule{4.5}{5.5}}%
\fi
\makeatother

\newcommand{\balita}{\raisebox{-0pt}{\text{\nano$\bullet$\hspace{.75pt}}}}

\usepackage{tikz}
\usetikzlibrary{automata,arrows,topaths}
\usepackage{caption}

\usepackage[bookmarks=false]{hyperref}
\hypersetup{
         colorlinks   = true,
         citecolor   =azulESI,
        urlcolor = azulESI
}
\hypersetup{linkcolor=azulESI}

    \newcommand{\Gr}[1]{\mathscr{G}_{#1}}
        \newcommand{\Mr}[1]{\mathscr{M}_{#1}}

    \newcommand{\vuno}{\raisebox{-.3\height}{\includegraphics[height=2.3ex]{Item4_V1v.pdf}}}
    \newcommand{\vdos}{\raisebox{-.3\height}{\includegraphics[height=2.3ex]{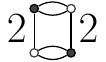}}}
    \newcommand{\vtres}{\raisebox{-.3\height}{\includegraphics[height=2.3ex]{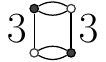}}}

\newcommand{\itemb}{\item[\balita]}

\usepackage{float}
\usepackage{subcaption}
\usepackage{amsthm}
\numberwithin{equation}{section}
\newtheoremstyle{mytheoremstyle} 
    {10pt}                    
    {8pt}                    
    {\itshape}                   
    {}                           
    {\scshape}                   
    {.}                          
    {.5em}                       
    {}  

\makeatletter
\newcommand{\leqnomode}{\tagsleft@true}
\newcommand{\reqnomode}{\tagsleft@false}
\makeatother
\theoremstyle{mytheoremstyle}

\newtheorem{theorem}{Theorem}[section]
\newtheorem{theoremconjecture}[theorem]{Theorem-Conjecture}
\newtheorem{corollary}[theorem]{Corollary}
 \newtheorem{lemma}[theorem]{Lemma}
 
 \newtheorem{proposition}[theorem]{Proposition}
 \newtheoremstyle{definition} 
    {8pt}                    
    {5pt}                    
    {}                   
    {}                           
    {\scshape}                   
    {.}                          
    {.5em}                       
    {}  

 \theoremstyle{definition}
 \newtheorem{definition}[theorem]{Definition}
 \newtheorem{example}[theorem]{Example}
 
 \newtheorem{remark}[theorem]{Remark}

\newcommand{\Z}{\mathbb{Z}}

\newcommand{\C}{\mathbb{C}}
\newcommand{\R}{\mathbb{R}}
\newcommand{\runter}[1] {\raisebox{-.45\height}{#1}}

\newcommand{\dif}{{\mathrm{d}}}

\newcommand{\uni}{\mathrm{U}}

\newcommand{\ii}{\mathrm{i}}
\newcommand{\ee}{\mathrm{e}}

\newcommand{\T}{\mathcal{T}}

\DeclareMathOperator{\Tr}{Tr}

\newcommand{\pimax}{\pi_{\max}}

\newcommand{\hp}[1]{^{(#1)}}

\newcommand{\HN}{M_N^{\text{\tiny s.a.} }(\C)}
\usepackage[includeheadfoot,margin=2.7cm]{geometry}
\usepackage[font=small,labelfont=bf,tableposition=top]{caption}
\usepackage{multicol,caption}

\tikzcdset{arrow style=tikz, diagrams={>={Stealth[round,length=4pt,width=4.95pt,inset=2.75pt]}}}

 \title[Twofold universality of melonic random tensors]{Twofold universality of large-$N$\\melonic random tensors}
 \author[ C. I. P\'erez S\'anchez]{Carlos I. P\'erez S\'anchez}

  \address{}
  \email{\href{mailto:perez.sanchez@protonmail.ch}{perez.sanchez@protonmail.ch}}

\usepackage{stackrel}
\newcommand{\gauss}{\dif \mu_0 (T)}

\newcommand{\Wick}{\text{\tiny Wick\,}}

\newcommand{\conn}{^{\text{\tiny conn.}}}
\makeatletter

\newcommand*\notocchapter[1]{%
  \if@openright\cleardoublepage\else\clearpage\fi
  \thispagestyle{empty}\global\@topnum\z@
  \@afterindenttrue
  \let\@secnumber\@empty
  \@makeschapterhead{#1}\@afterheading
}

\makeatother
\newcommand{\und}{\text{and }}
\renewcommand{\forall}{\text{for all }}
\newcommand{\with}{\text{with }}

\begin{document}
\begin{abstract} We construct a measure
that exhibits two aspects of a new type of universality and dramatically simplifies
the integration of tensors $T_{a_1,a_2,\ldots,a_D} \in \C$ ($a_1,\ldots,a_D=1,\ldots,N$) at large $N$.
In contrast to matrix integration, in which matrix traces canonically
yield the integrand, tensors need additional information (equivalent to a $D$-coloured graph $B$) to contract their indices
and form a tensor trace $B(T)$.
We show that, whenever each
$B_1,\ldots, B_n$
can be obtained by a recursive construction known as melonicity,
then the leading order in $N$ of the integral of
 $ {B_1}(T) {B_2}(T) \cdots {B_n}(T) $
is independent of the---often intricate---combinatorics of the traces $B_i$,
but also, to our surprise, independent of $D$ as far as $D\geq 3$. Instead, at large $N$, these integrals are some functions (indexed by $n$) of the number of vertices $2p_i$ of $B_i$  which
we call \textit{melonic polynomials}. Melonic traces cumulants with respect to
any (`interacting') measure
\[\notag
\exp\Big\{-N^{D-1} \sum_{i=1}^m g_i {B_i}(T)\Big\}
\gauss \quad (g_1,\ldots,g_m \in \R, \gauss =\text{the tensor Gaussian}) \notag\]
with each $B_i$ melonic, can be computed with our universal measure
that replaces each $B_i$ by a canonical trace depending only on $p_i$.
We prove that any two \textit{melonic} tensor models are   indistinguishable at large-$N$,
independently of the number of tensor indices (first universality aspect), and of
the fine-grainedness of their interactions
(second universality), being a sufficient
condition that the couplings (the parameters $g_i$ above) agree
and their respective traces are monomials with the  same
degree in $T$.

\end{abstract}
 \maketitle%

\noindent
\section{Introduction}
Among the diverse ways to motivate tensor integrals,
the following example on matrix integrals is one of the briefest and closest to the
combinatorial context of the present article.
\begin{example}
Consider all possible gluings of the sides of \textit{one} $2p$-agon.
One allows to match the polygon's sides
pairwise, as depicted next with chords for a rooted hexagon:
\[\runter{\includegraphics[width=.44\textwidth]{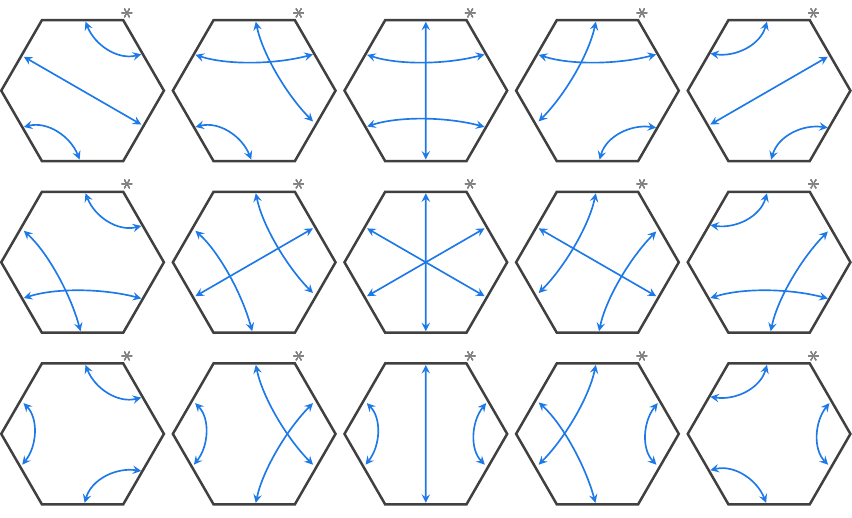}} \label{chords}
\]
All these $2p$-agons have only one face and, after identifications, $p$ edges;
however, the number of vertices $V$ of each gluing depends on the chord diagram.
For instance,
in Diagrams \ref{chords} above, the first row's numbers  of vertices (from L to R) are $V=4,2,2,2,4$.
The genus  $\mathfrak g$ of a such
gluing is given by the Euler formula
$2-2\mathfrak g=V-p+1$.
One might wish to determine the number $c_{ \mathfrak g }(p)$ of genus-$\mathfrak g$
gluings obtained from a $2p$-agon. These are genus-$\mathfrak g$ Catalan numbers,
as the planar ones are the non-crossing chord diagrams.
Interestingly for us, their generating series  (or polynomial)
$Q_p$
have a  matrix integral representation
\[
Q_p(N)& := \frac{1}{N^2}\sum_{\mathfrak g\geq 0 } c_{ \mathfrak g }(p)  N^{2-2\mathfrak g} =
 \frac1N \int_{\HN}
 \Tr (H^{2p}) \dif\mu_{\text{\tiny matrix}} (H) \notag
\]
on the space of Hermitian $N\times N$-matrices ($ \dif\mu_{\text{\tiny matrix}} (H)$
having correlation
$ \mathbb E_0  [ {H_{a,b}H_{c,d}} ]  = \frac{1}{N}\delta_{b,c}\delta_{a,d} .$)
For Diagrams \ref{chords}, this means $Q_3(N)=5+10/N^2$. It is a
result of Harer-Zagier \cite{HarerZagier} that the sum of these polynomials $Q_p$
yield a rational function\footnote{Incidentally,
this very Harer-Zagier rational function enumerates
the volumes of the $L_1$-spheres in the $\Z^N$ lattice, as argued in
\cite[Sec. 4]{QuiversNCG}.} \cite{HarerZagier}
\[ \notag
  1+ 2z N + 2 z \sum_{p\geq 1 }
 \frac{  Q_p(N) }{(2p-1)!!} (Nz)^{p} = \bigg(\frac{1+z}{1-z}\bigg)^N
\]
that determines $c_{\mathfrak g}(p)$ for all genera and number of sides.
\hspace{30ex}$\diamond$
\end{example}
\noindent
The type of tensor integrals we
address below started in \cite{SasakuraTM,AmbjornTM,Nexpansion_coloured} with the motivation of
describing higher-dimensional and multi-cellular analogues of the previous example. Here, we will not use physics or employ physical language (except in applications, Sec. \ref{sec:applications} and the outlook \ref{sec:phys_outlook}), and refer to \cite{Gurau:2024nzv} for an updated report.
While it would be interesting to address tensor integrals
in an analogous spirit to the Harer-Zagier formula above,
we do not pretend here to focus on this enumerative geometric meaning of tensor
integrals either (these aspects are treated in detail in
\cite{arXiv:2212.12200} and references therein). Initially we cared about
mastering the computation of integrals and, in so doing,
we discovered a simplification of the large-$N$ limit of a
certain kind of integrands traditionally called \textit{melonic} (see Sec. \ref{sec:Melons}).
The universality of tensor integrals
(Secs. \ref{sec:TM}, \ref{sec:MelPolyn}) is also seen to yield
a criterion of equivalence  of two melonic random tensor models
(under the conditions of Thm. \ref{thm:interactive_theory_equivalence_D}) at large-$N$.
We keep this article self-contained.

\tableofcontents

\allowdisplaybreaks[2]
\subsection{The integrand}\label{sec:integrand}
In this article we fix the meaning of two integers $N \geq 2 $ and $ D \geq 3$, as parameters of the space of tensors $T: (\C^N)^{ \otimes D} \to \C$ that will
build our integrand. This integrand obeys the following conditions: Each one of the $D$ factors of the unitary group $\uni (N)^D$  (with $\uni (N) = \mathcal U(\C^N)$ the unitary group of $\C^N$)
acts  \textit{independently} on the vector space $ (\C^N)^{ \otimes D} $. The independence of this group action can be taken as an axiom and interpreted as the tensor's entries not satisfying symmetries under exchange of indices (this framework started in
\cite{Gurau:2009tw} and evolved to \cite{critical,uncolouring}). More precisely, for each $\mathbf U  = (U\hp 1,\ldots,U\hp D ) \in \uni (N)^D  $ the explicit action is
\[
T^{\mathbf U}(X_1,\ldots, X_D)= T(U\hp 1 X_1,U\hp 2 X_2,\ldots, U\hp D X_D), \quad \with X_1,\ldots, X_D \in \C^N,
\]
being $U\hp c  X_c$ given by the defining representation for each $c=1,\ldots,D$ (so $U\hp cX_c$
is just a matrix multiplying a vector). We integrate unitary invariant, $\C$-valued
monomials $B(T )$ in $T$, that is
\[B: (\C^N)^{\otimes D}\to \C \quad \with B(T^{\mathbf U} )= B(T ). \]
Other sources write $B(T,\overline T)$ as function of both arguments $T$ and $\overline T$,
but both conventions mean the same.\par

Let us exhibit these unitary invariant monomials---henceforth  called just \textit{invariants}---in terms of the components $T_{i_1,i_2,\ldots,i_D} = T(e_{i_D},\ldots,e_{i_D})$ with respect to   $D$ bases $\{e_{i_c}\}_{i_c=1,\ldots,N}$ ($c=1,\ldots,D$) of the $D$  factors $\C^N$ of $ (\C^N)^{ \otimes D} $. Invariance requires the
appearance of the complex conjugate $\overline T_{i_1,i_2,\ldots,i_D}$ of $T_{i_1,i_2,\ldots,i_D}$, in terms of which the degree $\deg B(T)$ of $B(T)$ is the sum of the
degrees in $T$ and $\bar T$ (which must coincide, due to invariance).
Examples of invariants are
\[\begin{tabular}{c|c|c}
$D$ & $\deg B$ & $B(T)$  \\[.5ex]
\hline && \\[-1.8ex]
$3$ & $2 $  & $ \sum\limits_{i_1,i_2,i_3 \in \{1,\ldots,N \} }%
T_{i_1,i_{2},i_{3}}%
\bar{T}_{i_{1},i_{2},i_{3}}
$ \\
$3$ & $6 $  & $ \sum\limits_{\mathbf{i},\mathbf{n},\mathbf{x} \in \{1,\ldots,N\}^{3}}%
T_{i_1,i_{2},i_{3}}T_{n_1,n_{2}n_{3}}T_{x_1,x_{2},x_{3}}%
\bar{T}_{i_{1},x_{2},n_{3}} \bar{T}_{n_{1},i_{2},x_{3}} \bar{T}_{x_{1},n_{2},i_{3}} $ \\
$6$ & $2 $  & $\sum\limits_{\mathbf{i} \in \{1,\ldots,N\}^{6}}%
T_{i_1,i_{2},i_{3},i_{4},i_{5},i_{6}}%
\bar{T}_{i_{1},i_{2},i_{3},i_{4},i_{5},i_{6}}  $ \\
$6$ & $4 $  & $\sum\limits_{\mathbf{i},\mathbf{n} \in \{1,\ldots,N\}^{6}}%
T_{i_1,i_{2},i_{3},i_{4},i_{5},i_{6}}T_{n_1,n_{2},n_{3},n_{4},n_{5}n_{6}}%
\bar{T}_{i_{1},i_{2},i_{3},n_{4},n_{5},n_{6}} \bar{T}_{n_{1},n_{2},n_{3},i_{4},i_{5},i_{6}}  $ \\
\end{tabular}
\]
It is usual to dispense with the indices and economically encode the invariants
otherwise \cite{GurauRyan}. In terms of graphs, the invariants of the previous table read, respectively,
\[
\runter{\includegraphics[height=1.7ex]{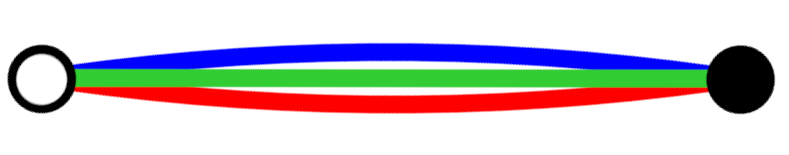}} \qquad
\runter{\includegraphics[height=12ex]{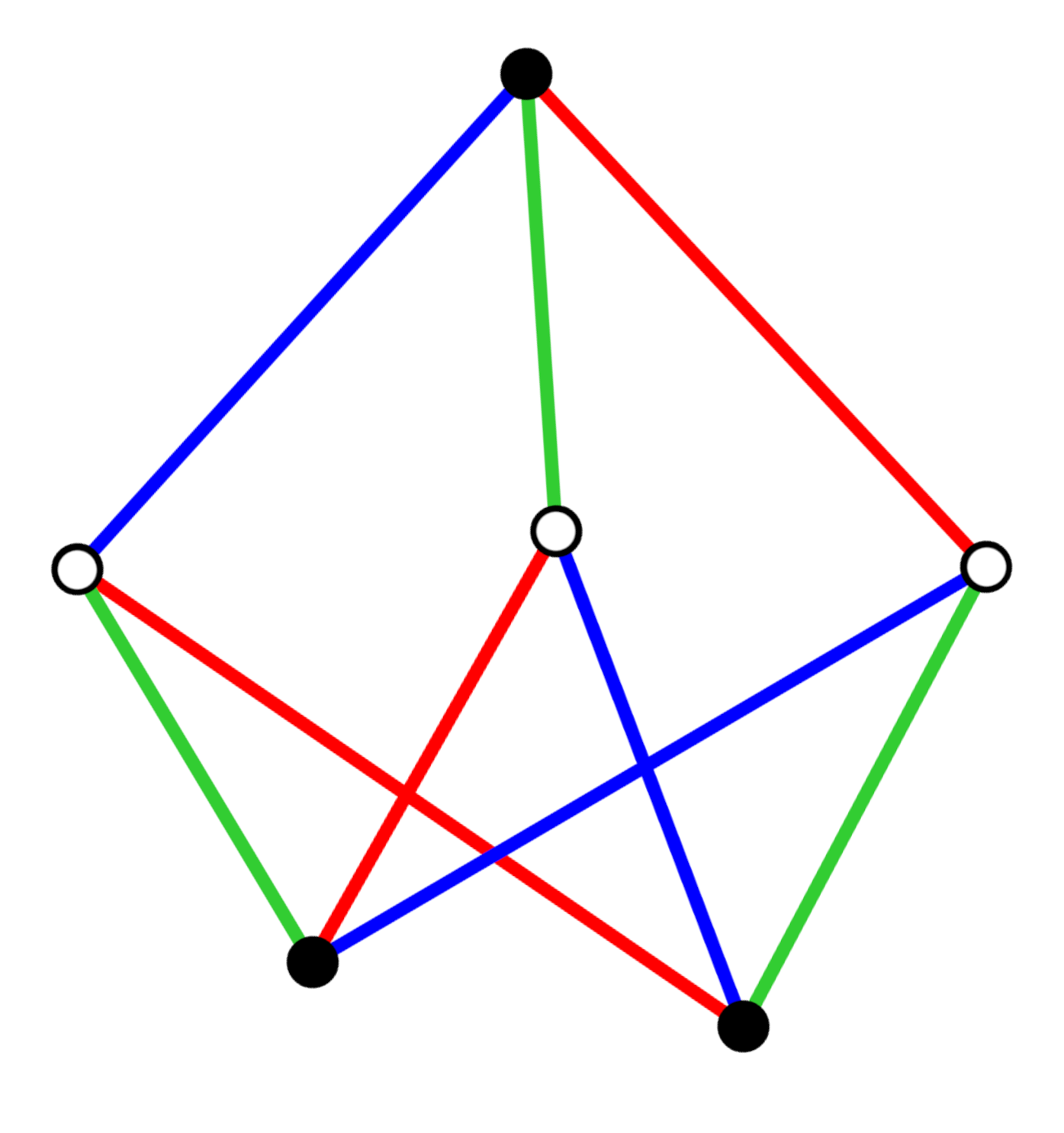}} \qquad
\runter{\includegraphics[height=2.3ex]{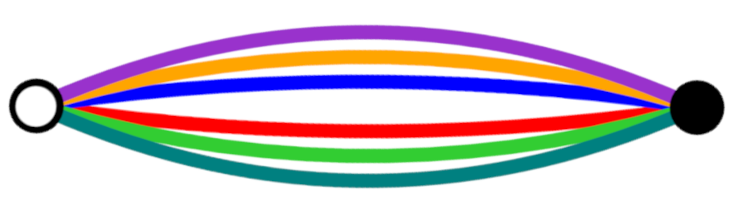}} \qquad
\runter{\includegraphics[height=10ex]{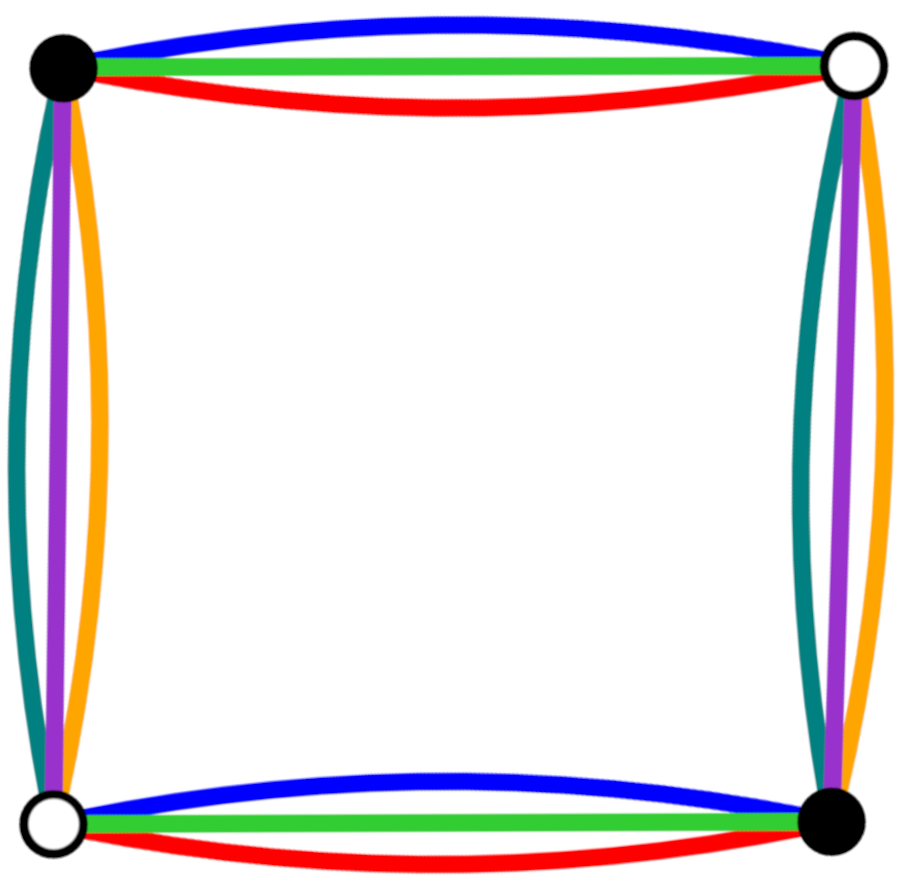}}
\]
Here, each occurrence of
$T$ in a monomial corresponds to a white and each of $\bar T$
to a black vertex, both with $D$ incident
half edges of $D$ different \textit{colours} $\{1,2,\ldots,D\}$. For each index-match
$i_c$ in $T_{\cdots i_c \cdots}$ and $\bar T_{\cdots i_c \cdots}$
the $c$-th coloured half edges are joined.  We thus obtain a map between the $D$ ordered tensor indices
which regularly colour the $D$ edges adjacent at each vertex, which turns out to extend to a bijection \cite{GurauRyan}
\[ \nonumber
\{\text{invariants $B:(\C^N)^{\otimes D} \to \C$} \}  \leftrightarrow \{\text{regularly edge-$D$-coloured  vertex-bipartite graphs}\}.
\]
The mouthful set on the right is very often conveniently called `$D$-coloured graphs'. By grace of this map, `invariant' will unambiguously refer to either the invariant polynomial or to the graph it corresponds to. These graphs, and therefore the invariants,
are enumerated in \cite{counting_invariants}.
\\

We denote by $
\dot \cup_i B_i$ the disconnected graph
with connected components $B_i$ (we write the disjoint union
`$\dot \cup$' and not just union `$\cup$'
since some connected components might coincide). The invariant
associated to $
\dot \cup_i B_i$ is $\prod_i B_i(T)$.
A common object appearing below is a graph with an additional
colour labelled $0$ (when drawn, represented in gray in the sequel), which
arises from Wick contracting graphs.
In this context, a \textit{Wick contraction or Wick pairing} of a coloured graph  $B$
is a bijection between the white and the black vertices of $B$.
We denote by $\pi(B)$
the graph that arises from a $D$-coloured
graph $B$ with the additional $0$-coloured edges defined
by the pairs present in $\pi$. Thus $\pi(B)$ is a
$D+1$-coloured graph.
Since there will always be some $D$ fixed by the context,
conventionally $(D+1)$-coloured graphs will
always have edge colours $\{0,1,\ldots,D\}$.\\

\noindent
Given a coloured graph $B$ and a Wick pairing $\pi$ of $B$, a \textit{face} of the graph
$\pi(B)$ is a connected subgraph of $B$ with edges of colours $0$ and $c$ (by regularity of the colouring, necessarily alternating) for some $c=1,\ldots,D$.  We denote by
$F(G)$ the set of faces of any $(D+1)$-coloured graph $G$. By $V(\gamma)$
we shall denote the set of vertices of any graph $\gamma$.
The convention for ordering the entries of
an edge $e=(v,w)$ of a vertex-bipartite graph
is that $v$ is black (or, if enumerated, even) and $w$ white (or odd).
\\

\noindent
The tensor size $N$ still has not played any role, but it
prominently appears below in the integration of these
arrays $T_{i_1,\ldots,i_D}$ of size $N\times N \times  \ldots \times N $ ($D$-tuple factor).

\subsection{Tensor integrals}\label{sec:integrals_intro}
Each entry $T_{i_1,i_2,\ldots,i_D}$ being a complex variable, the integration measure
on $[(\C^N)^{\otimes D}]^* \simeq (\C^N)^{\otimes D}$ is over $2N^{D}$ real variables
\[ \dif T  =  \prod_{i_1,\ldots,i_D=1}^N \bigg( N^{D-1}
  \frac{ \dif T_{i_1,i_2,\ldots,i_D}\dif \overline T_{i_1,i_2,\ldots,i_D}}{2\pi \ii} \bigg)   \]
The factor $N^{D-1}$
guarantees the normalisation of the Gaussian measure $\gauss$ on $ (\C^N)^{\otimes D}$
defined by its expectation being
\[
\E_0 \hp N[
T_{i_1,i_2,\ldots,i_D} \overline T_{j_1,j_2,\ldots,j_D} ]
=
\frac{1}{N^{D-1}}
\delta_{i_1,j_1}
\delta_{i_2,j_2}
\cdots\delta_{i_D,j_D}
\]
or, in terms of\,\footnote{We will oversimplify the notation a little bit:
although the Gaussian $\gauss$ depends on $D$, since
it will appear always with a prefactor of the type $\exp\{-N^{D-1} S(T)\}$
for some $\R$-valued function $S(T)$, we will omit $D$, which
can be read off.} $\dif T $, by $\gauss=\ee^{- N^{D-1} \sum_{i_1,i_2,\ldots,i_D=1}^N T_{i_1,i_2,\ldots,i_D} \overline T_{i_1,i_2,\ldots,i_D}} \dif T$. So
\[
\E_0\hp N[ B(T) ]
=
\int_{(\C^N)^{\otimes D}}
B(T)
\exp\bigg(- N^{D-1} \sum_{i_1,i_2,\ldots,i_D=1}^N T_{i_1,i_2,\ldots,i_D} \overline T_{i_1,i_2,\ldots,i_D} \bigg)
\dif T.
\]

\par

Given connected $B_i$ invariants one can compute, using Wick or Isserlis Theorem, 
\[
\int_{(\C^N)^{\otimes D}} B_1(T)B_2(T) \cdots B_n(T)  \gauss =
\sum_{\substack{\text{\tiny Wick contractions} \\
\pi\,\, \text{\tiny of } \dot \cup_i B_i} }
A(G_\pi),\]
where, $G_\pi$ abbreviates the $D+1$-coloured graph
$ \pi(\dot \cup_i B_i)$ and, for any
graph with edge colours $\{0,\ldots, D\}$, we define
the \textit{amplitude} of $G$
(denoting cardinality by $\#$) by
\[\label{amplitude}
A(G) :=
N^{\#F(G)-\frac{D-1}{2}\#V(G) } \,.
\]
Notice that $A(G)$ is monic, while
most conventions include symmetry prefactors in their definition of the amplitude.
The difference is consistent with the fact
that we sum above over Wick
contractions and not over graphs (the several
isomorphic graphs that arise from different
Wick contractions will yield the equivalence).
%
We also define the connected integral $\int\conn$ by
restricting the sum to only connected graphs $\pi(\cup_i B_i)$,
\[\label{integral_def}
\int_{(\C^N)^{\otimes D}}\conn B_1(T)B_2(T) \cdots B_n(T)  \gauss :=
\sum_{\substack{\text{\tiny Wick contractions} \\
\pi\,\, \text{\tiny of } \dot \cup_i B_i \\
G_\pi =\pi [ \dot\cup_{i=1}^n B_i ] \text{\tiny connected}
}}
A(G_\pi).\]
Several essential contributions to tensor models
rely on rewriting the amplitudes to show its $1/N$ expansion \cite{Nexpansion}. From it, we only will care in the present
article about Wick contractions
that maximise faces, and shall not rely directly
on Gur\u au's deep results. An advantage of doing so,
is to keep a brief and self-contained presentation of this article.

\section{Melonic graphs}\label{sec:Melons}

A \textit{dipole $\delta$ of colour $c$}
is a graph with two vertices joined by $D-1$ edges
differently coloured by the set $\{1,2,\ldots, c-1,c+1,\ldots, D\}$,
and one colour-$c$ half-edge at each of the two vertices.
The  \textit{insertion of a dipole $\delta$
 of colour $c$} at a $c$-coloured edge $e$ of a
$D$-coloured graph $B$, is the only replacement (out of two)
of $e$ by $\delta$ inside $B$ that yields a $D$-coloured graph.

By definition, a \textit{melon} (or a \textit{melonic graph}) is a coloured graph $B$ that is constructed
from a regular $D$-ary edge-$D$-coloured rooted tree $\mathcal T$ in the following way.
Let $p$ be the number of vertices that are not leaves ($p$ is the number of parents).

\begin{itemize} \itemb First, enumerate all the parents of $\mathcal T$ by even numbers $\{0,2,\ldots,2p-2\}$ respecting the following restriction:
any numbered vertex must have a larger label
than its parents.

 \itemb Let
$B\hp0 $ be the quadratic observable $T\cdot \overline T$ and
tag one of its
vertices by $0$ (the other by $1$).

\itemb For $n=1,\ldots, p-1$,
$B\hp n(\mathcal T)$ is the invariant that arises from $B\hp{n-1}(\mathcal T)$ by insertion of a colour-$c$
dipole inside the $c$-coloured edge at the vertex numbered with $2(n-1)$.
Label the vertices of the new dipole by $2n$ and $2n+1$, else keeping old vertex labels.

\itemb The final invariant $B=B(\mathcal T)$ is obtained from  $B\hp {p-1}(\mathcal T)$  after replacing the vertices
with even (resp. odd) parity
by black (resp. white) vertices.
\end{itemize}

\begin{example} For $D=3$, let us see how the  trivalent tree
\[
\mathcal T = \runter{\includegraphics[height=10ex]{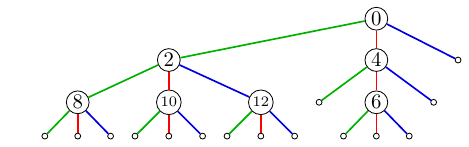}}
\quad \text{ yields  }
B(\mathcal T) = \runter{\includegraphics[height=10ex]{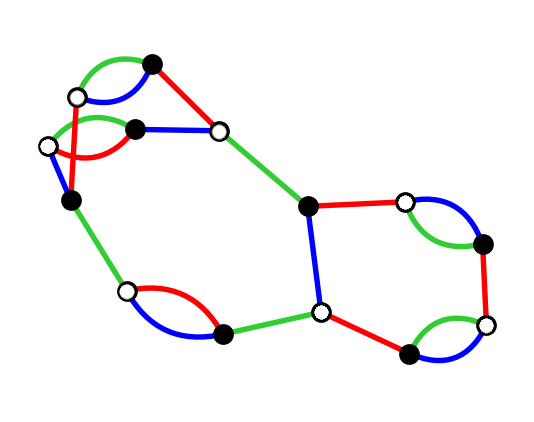}}
\]
by exhibiting the steps:\\
\[
\begin{tabular}{clllllll}
\runter{\includegraphics[height=5ex]{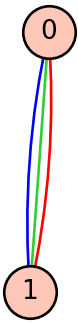}} \!\!\!\! & $\rightarrow$ & $B\hp 1(\mathcal T)=$
\runter{\includegraphics[height=3.9ex]{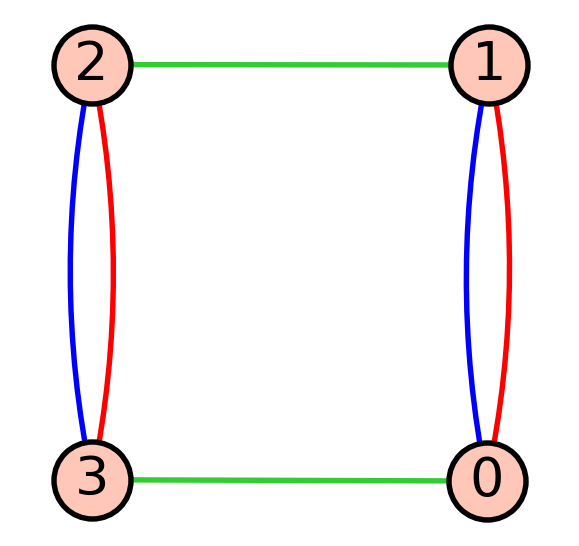}}  & $\to$  & $B\hp 2(\mathcal T)=$
\runter{\includegraphics[height=5ex]{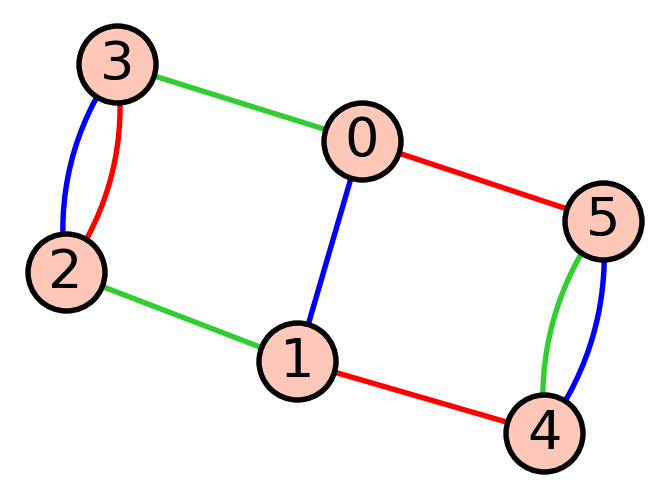}}
  & $\rightarrow$ & $B\hp 3(\mathcal T)=$
\runter{\includegraphics[height=6ex]{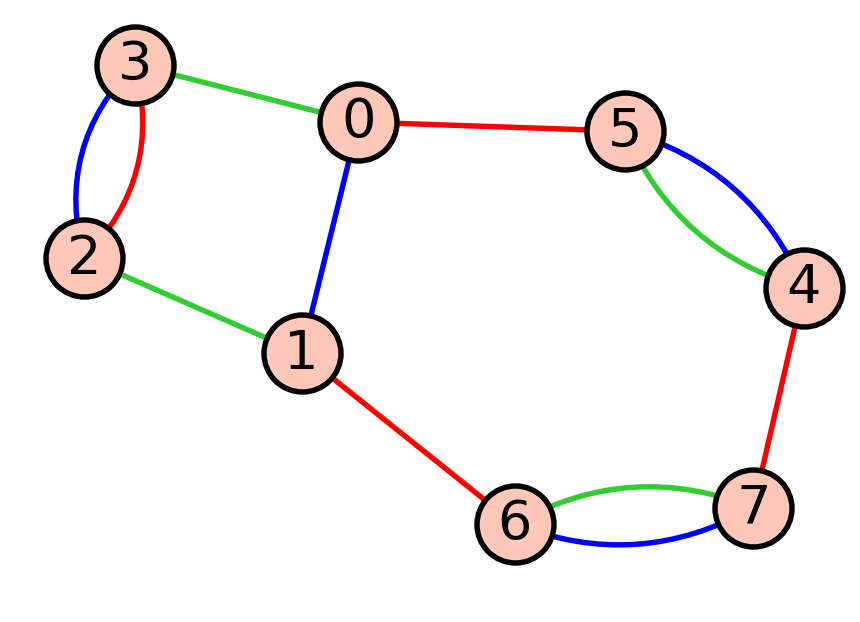}}  \\[4ex]
& $\to$  & $B\hp 4(\mathcal T)=$
\runter{\includegraphics[height=7ex]{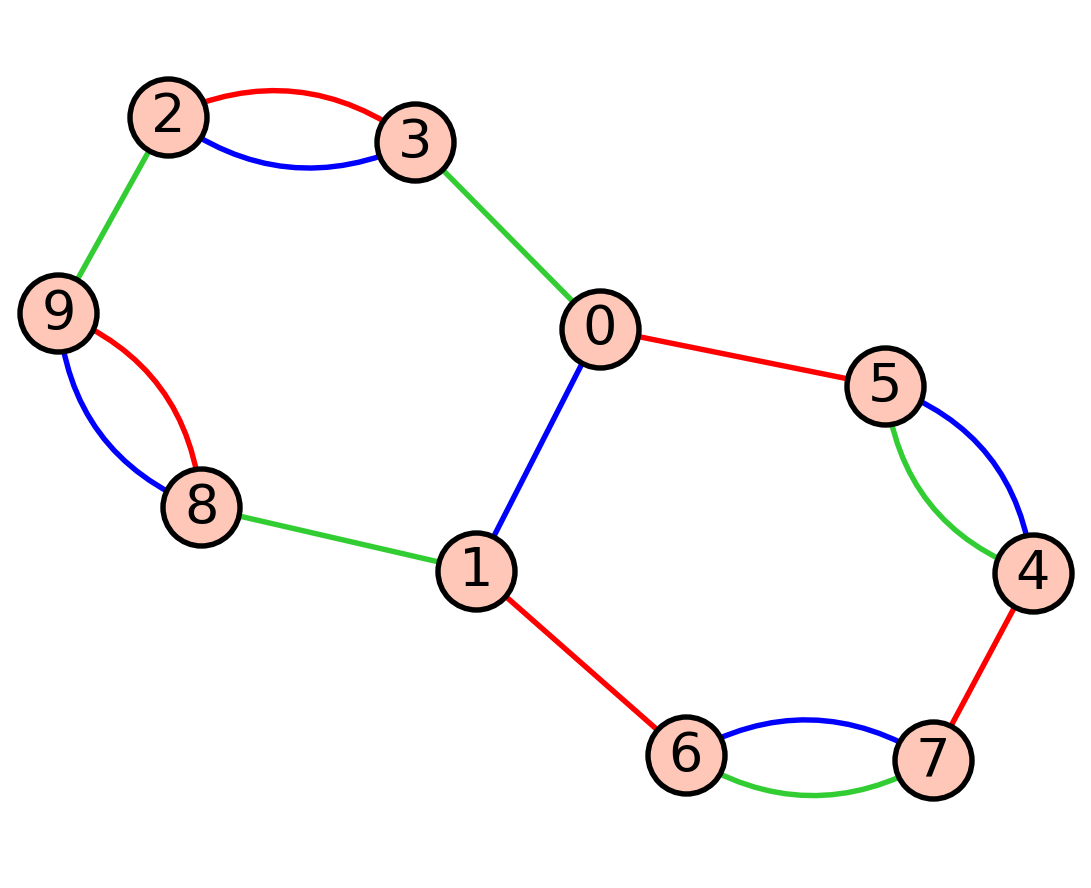}}
  & $\rightarrow$ & $B\hp 5 (\mathcal T)=$
\runter{\includegraphics[height=10ex]{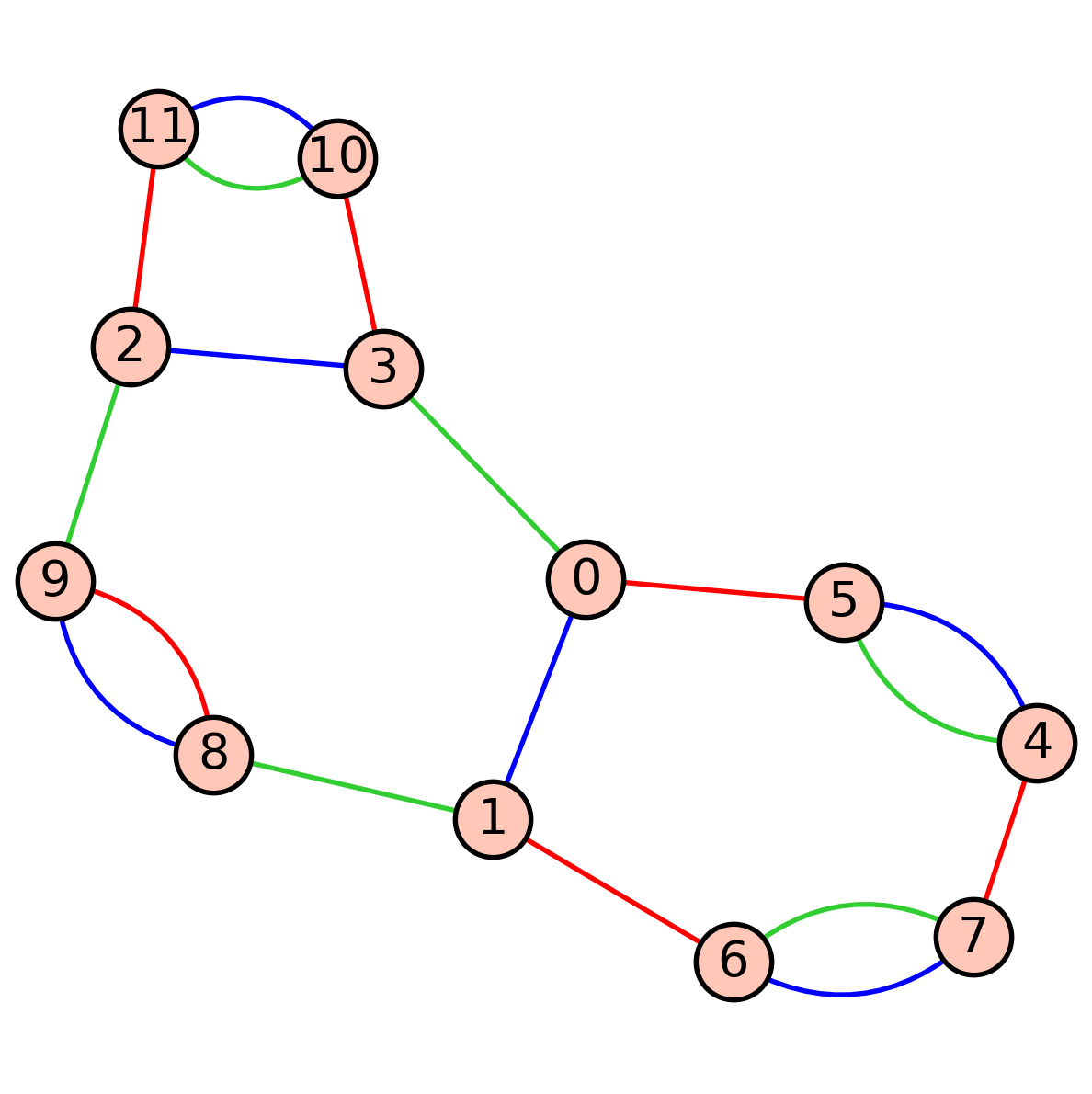}}  & $\to$  & $B\hp 6(\mathcal T)=$
\runter{\includegraphics[height=10ex]{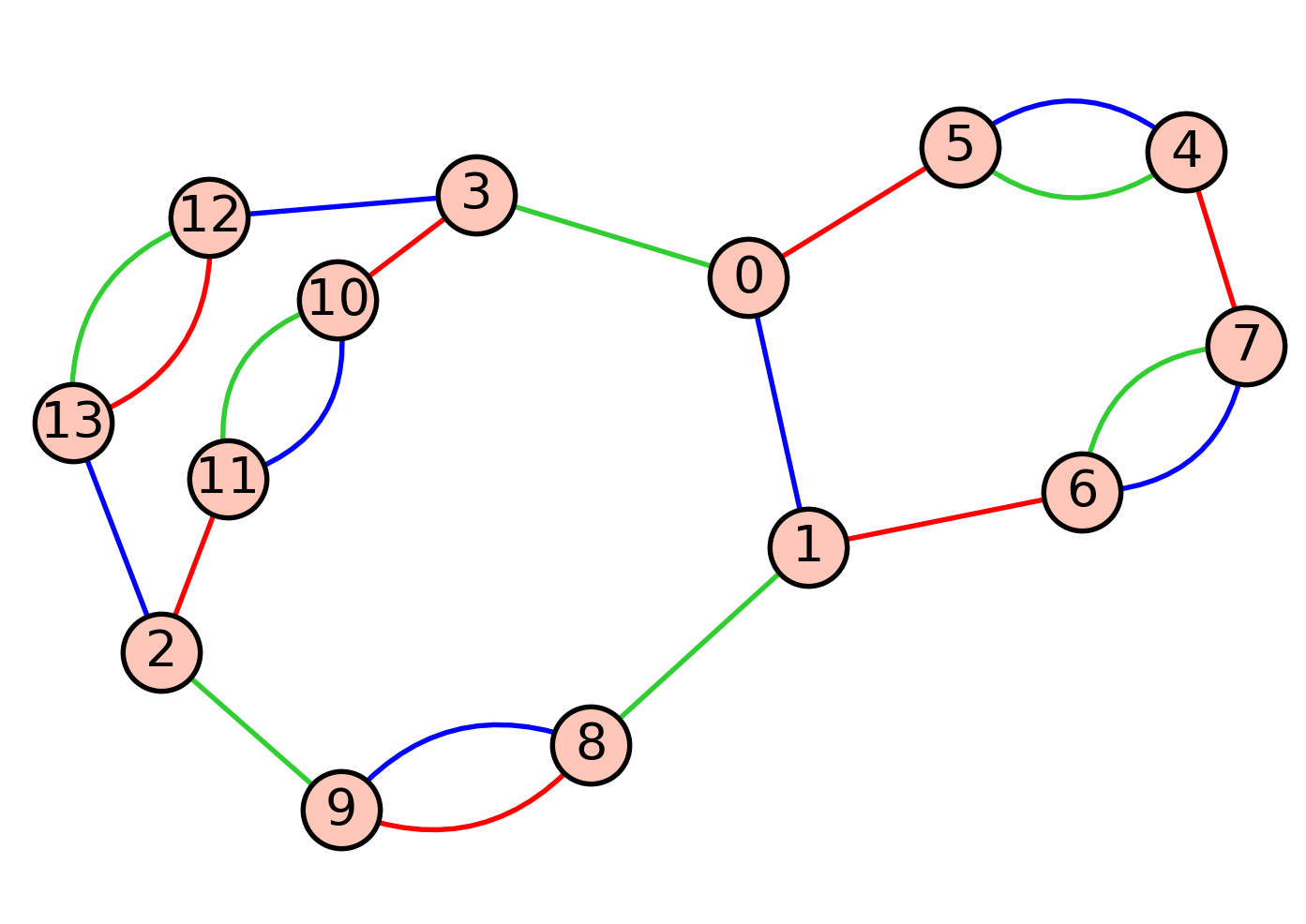}}
\notag
 \end{tabular}\]
and then colouring the vertices by the parity of their label.
\end{example}


\section{Universality of the leading order of melonic integrals}\label{sec:TM}

Observe that the number of faces induces a preorder in the set of Wick pairings of a
fixed graph $B$, namely defined by
\begin{center}
$\pi \leq \pi'$ if and only if $\# F(\pi(B)) \leq \#  F(\pi'(B)) $.
\end{center}
(While we should write this
more transparently, $\pi \leq_B \pi'$, we prefer a light notation. A second caveat is the lack of `antisymmetry', $\pi\leq \pi'$ and $\pi'\leq \pi$
do not imply equality.) A Wick contraction $\pi_1$ is \textit{maximal}
if \[ \# F (\pi_1(B)) = \max \{ \#F(\pi(B)) : \text{$\pi $ is a Wick contraction of $B$ }  \}, \] and there are, in general, several maximal elements. The rest of this section shows that this cannot happen if
$B $ is connected and melonic.
\subsection{Single trace integrals}

\begin{lemma}[Maximal Wick contractions at a dipole]
\label{lem:MaxWick}
For $D\in \Z_{\geq 2}$, let $B$ be a $D$-coloured graph  and $\pi$ a Wick contraction of $B$.
Let $\delta\subset B$ be a dipole of $B$.
If $\pi$ is a maximal Wick pairing, then the two vertices $(v,w)$ of the dipole $\delta$ inside the graph $\pi(B)$ must be
Wick-contracted by $\pi$, i.e. $(v,w) \in \pi$.
\label{lem:max_faces}
\end{lemma}

\begin{proof}
Let $(v,w)$ be the vertex pair of the dipole $\delta$, let denote by $c$ the colour of $\delta$; let $x$ and $y$ be, respectively, the vertices connected to $v$ and $w$ by a $c$-edge (so the dipole
is attached to $B$ at $x$ and $y$). \textit{Suppose} that $\pi'$ is a Wick contraction of $B$
with $(v,w) \notin \pi'$ --- instead $(v,s), (r,w) \in \pi'$ for
vertices $r$ and $s$ of $B$ with $r\neq v$,
and $s\neq w$. \par
Let us denote by
$f_i'$ the $(0i)$-face of colour $i\neq c$ that contains $r$. Also for the rest of the proof, the prime will
stress that we refer to the faces of $\pi'(B)$.
We follow  three edges: first $(r,w)$, then the $i$-coloured edge at the dipole $\delta$,
and thirdly $(v,s)$. Then $s$ lies on the same face $f_i'$.

In contrast, for the colour-$c$ faces, we obtain a dichotomy, depending on
the answer to this question: \begin{center}
                              \textit{Is
there is a $ (0,c,0,c,\ldots,0,c)$ bicolored path
from $x$ to $r$? }
                             \end{center}
(A path is a concatenation of neighboring edges, which, in this
case, also respects the order of appearance of said colours, despite
the edges not being oriented.)
To answer the question in cases, let us denote by
$e_c(v)$ and $e_c(w)$ the $c$-coloured edges of the dipole $(v,w)$.

\begin{itemize}
\itemb \textit{Case `Yes'.} Let $f_c'(r)$ be the
$0c$-face containing both $x$ and $r$. Then it does not contain $s$, as it $f_c'(r)$ must consists of
the path from $x$ to $r$ (which by assumption, exists)
and closes by adding the edges $(r,w)$ and
$e_c(w)$. Also, since the $c$-coloured-edge slot at $y$
is occupied by $e_c(v)$, the colour-$c$ face
$f_c'(y)$ through $y$
must contain $s$ too.
Denote by $\Phi'$ the faces of $\pi'(B)$ that
do not contain the edges $(r,w)$ and $(v,s)$. We obtain
\[
F(\pi'(B)) = \{f_i'\}_{i=1,\ldots,D, i\neq c}  \cup\{f_c'(r), f_c'(s) \}  \cup \Phi'
\]
Now define $\pi$ by the Wick contraction that coincides
with $\pi'$
for all edges except
$(v,s)$ and $(r,w)$, which are replaced by
$(v,w)$ and $(r,s)$. Then
\[
\# F(\pi(B)) & =   (D-1)_{\text{colour $i\neq c$, at $v$ and $w$}} \\&
+(D-1)_{\text{colour  $i\neq c$, at $r$ and $s$}} \\
& + 1_{\text{colour $c$, containing $r,s,v,w$}} + \#\Phi.
\]
Similarly to the above,  $\Phi$ is the set of faces of $\pi(B)$ that
do not contain the new edges $(v,w)$ and $(r,s)$.

So subtracting the number of faces yields
\[
\# F(\pi(B)) -
\# F(\pi'(B)) = (2D-2 + 1+ \#\Phi) - (D + 1 + \#\Phi' )
\]But by definition of $\pi$,
$\Phi'=\Phi$ (all those faces coincide),
so
\[
 \# F(\pi(B)) -
\# F(\pi'(B)) = D-2 >0.
\]

\itemb \textit{Case `No'.} If no
$ (0,c,0,c,\ldots,0,c)$ path leads from $x$ to $r$, then a path (notice its ending)
$ (0,c,0,c,\ldots,0,c,0)$ connects $x$ and $y$.
This means that another $(c,0,c,\ldots,0,c)$-path (!)  connects $r$
with $s$.
It is easy to see that
then $F(\pi'(B))$ consists of $\Phi'$ and of $D$ faces,
since for each colour there a single face contains the six points
$r,w,x,y,v$ and $s$.  \par

\noindent
If $\pi$ is defined as in Case `Yes', there are $D$ faces
  containing $x,y,v,w$ and other $D$ faces
containing $r$ and $s$. Then
\[
 \# F(\pi(B)) -
\# F(\pi'(B)) = 2D-D >0.
\]

\end{itemize}

In both cases $\pi'$ cannot be maximal, if the dipole vertices are not Wick contracted.
\end{proof}

We now present the algorithm to construct a Wick contraction, whose maximality
we shall show later.
\allowdisplaybreaks[2]
Given a regular $D$-ary tree $\mathcal T $ let $B(\mathcal T)$ be its
invariant, we produce a Wick contraction $\pi$ of $B(\mathcal T)$ with
the Algorithm \ref{algorithm1}. An example follows. \\[1ex]
%

{\small
\begin{center}
\begin{algorithm}[t!]
\begin{minipage}{.75\textwidth}

    \SetAlgoLined
    \KwData{As input, a rooted regular $D$-ary tree $\mathcal T $, whose
    vertices are labelled by genealogy (children have a higher label than parents).}
    \KwResult{A Wick contraction associated to $\mathcal T$. }
     \textbf{Initialisation}:      $\pi=\emptyset$ and  $B(\mathcal T)=$
     coloured graph realisation of $\mathcal T$.

    \While{ $B(\mathcal T) $ has more than two vertices:}{
    \begin{itemize}
     \itemb Observe that there exists always (at least one group of) $D$ leaves in $\mathcal T$
with the same direct parent vertex.
Locate the corresponding dipole $\delta \subset B(\mathcal T)$ and replace it
by a $c$-coloured edge, being $c$ the colour of $\delta$ ($c$ is, by construction,
the colour of the ascendency line of the group of $D$ leaves towards its parent vertex in $\mathcal T$). Call the resulting graph $B(\mathcal T) \setminus \delta$.
\itemb Append  the two vertices $(v,w)$ of the dipole $\delta$
to $\pi$.
\itemb Redefine the graph $B(\mathcal T) \to B(\mathcal T) \setminus \delta$.
    \end{itemize}
    }
     Now that $B(\mathcal T)$ has two vertices $(v_0,w_0)$: \textbf{do}
     \begin{itemize}
      \item[]
      Add these two vertices to $\pi$, i.e. replace $\pi$ by $\pi \cup \{(v_0,w_0)\} $.
     \end{itemize}%
\noindent\textbf{Return} $\pi$.\newline

\caption{Producing a Wick contraction (later called $\pimax$) for a melonic graph from its  tree\label{algorithm1}.}
\end{minipage}
\end{algorithm}
\end{center}
} \normalsize

\begin{example}
For the tree $\mathcal T$ below we find the Wick pairing as prescribed by Algorithm
\ref{algorithm1} for its corresponding graph $B(\mathcal T)$:
\[
\mathcal T = \runter{
\includegraphics[width=6cm]{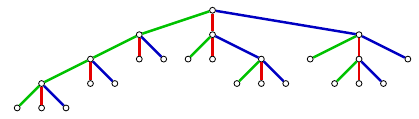}}\qquad\qquad
B(\mathcal T ) =
 \runter{
\includegraphics[width=3cm]{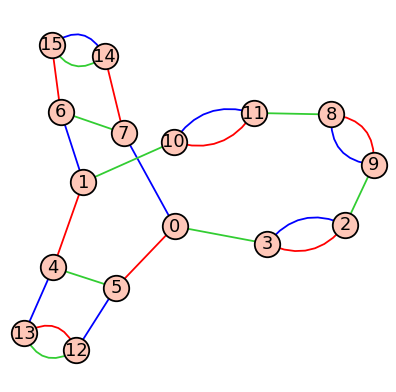}}
\]
\allowdisplaybreaks[2]
We signalise the dipole that is replaced by its colour
and the pair of vertices that are annihilated before arriving at the two-vertex graph:
\[\notag
\begin{tabular}{cccc}
  \runter{
\includegraphics[width=2.5cm,height=3.5cm,keepaspectratio]{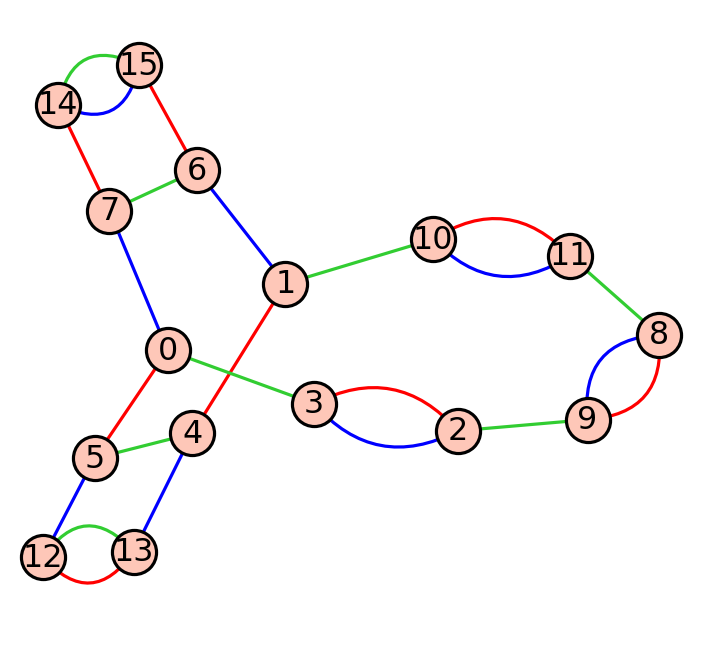}}
&   \runter{
\includegraphics[width=2.65cm,height=2.65cm,keepaspectratio]{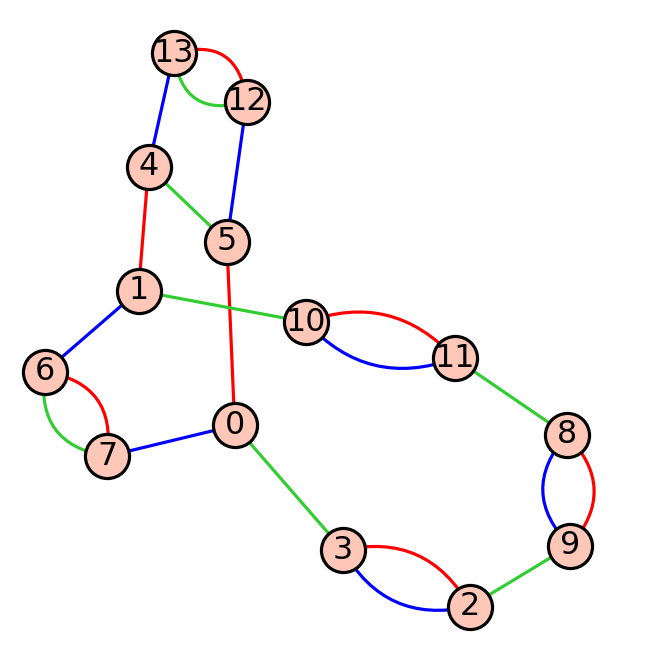}}
&
  \runter{
\includegraphics[width=2.390cm,height=3.8015cm,keepaspectratio]{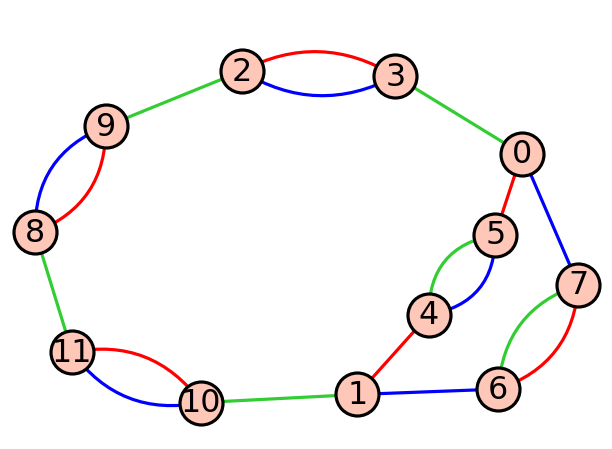}}
&
\runter{
\includegraphics[width=2.025cm,height=2.9025cm,keepaspectratio]{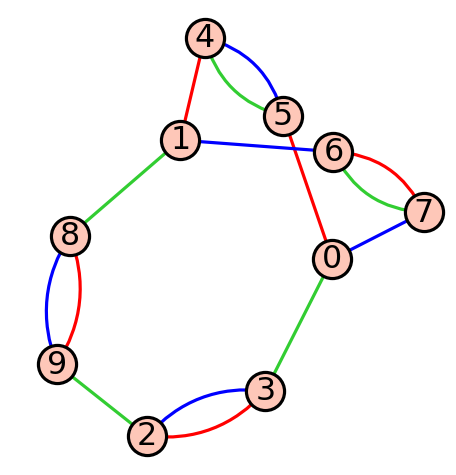}}
 \\
 $(14,15)$ & $(12,13)$ & $(10,11)$ & $(8,9)$
 \\[1ex]
 \runter{
\includegraphics[width=2.015cm,height=2.015cm,keepaspectratio]{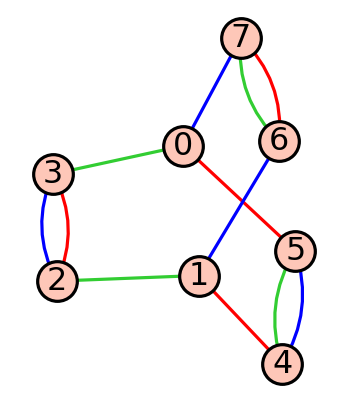}}
&
\runter{
\includegraphics[width=1.685cm,height=1.685cm,keepaspectratio]{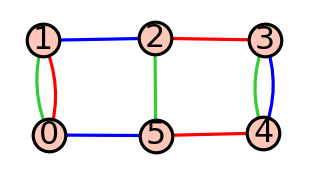}}
&\runter{
\includegraphics[width=1.15cm,height=1.5cm, keepaspectratio]{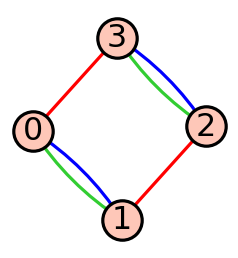}}
&\runter{
\includegraphics[width=.395cm,height=1.5cm, keepaspectratio]{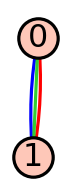}}
 \\
 $(6,7)$ & $(4,5)$ & $(2,3)$ &
\end{tabular}
\]
Thus the Wick contraction by the Algorithm \ref{algorithm1} is
$\pi = \{(2i,2i+1) : i =0,1,\ldots,7 \}$.
\end{example}

For clarity, we explain Algorithm \ref{algorithm1} in an equivalent way.
Given a melonic graph $B$, we produce a Wick contraction $\pi_{\max}$
of $B$ as follows. Let $\mathcal T$ be a tree that serves to construct $B$
and denote, as above, by $B\hp n(\T)$ the $n$-th insertion of a dipole
into $B\hp {n-1}(\T)$ for $n>0$, as dictated by the
labels of the tree; else let $B\hp 0(\T)$ be the $D$-coloured graph with two vertices,  which we denote $(v_0,w_0)$.
Let $(v_i,w_i)$ denote the new vertex pair in $B\hp i$, $\{(v_i,w_i)\} =
V(B\hp i) \setminus V(B\hp{i-1})$, for $0<i\leq p-1$, being $2p=\#V(B)$.
\begin{definition}\label{def:pi_max}
With the vertices defined by the previous paragraph, we denote the Wick contraction
emerging there by
$\pi_{\max} = \{(v_0,w_0),(v_1,w_1),\ldots,(v_{p-1},w_{p-1})\}.$
\end{definition}

This notation is justified by Proposition \ref{propo:max_unique} below.

\begin{lemma}[Unique maximal contraction]
\label{lem:unique_ExtContr}
Let $\pi$ be a Wick paring of a connected $D$-coloured graph $B$ and  $\delta$
a $c$-dipole $(c\in \{1,\ldots, D\})$, whose vertices we shall denote by $\{v,w\}$.
Let $B^+ = B\cup_e \delta$ be the insertion of the
dipole $\delta$ at an edge $e$ of colour $c$. Then
\begin{center}
$\pi$  is maximal if and only if the Wick pairing
$\Pi$ of $B^+$ is maximal,
\end{center}
being $\Pi = \pi \cup \{(v,w) \}$.
\end{lemma}
\begin{proof}
Suppose that $\Pi'$ is a Wick contraction of $B^+$ that maximises faces.
Then by Lemma \ref{lem:MaxWick} both $\Pi$ and $\Pi'$
must contract the vertices of dipole $\delta$,
$\{v,w\} \in \Pi \cap \Pi'$.
This implies that no vertex of $B$ is connected
to one of the dipole neither by edges of  $\Pi'$
nor of $\Pi$,
and therefore that
\[\pi':=\Pi' |_{V(B)} = \text{ the restriction of $\Pi'$ to the vertices of $V(B)$},
 \]
is a well-defined Wick contraction of $B$. In this notation,
notice that $\pi=\Pi|_{V(B)}$. In each case,
the total of faces are those of the original graph
plus those at the dipole, namely one gets the relations
\[
\# F (\Pi' (B^+)) &= \#F(\pi'(B)) + D-1,   \\
\# F (\Pi (B^+)) &= \#F(\pi(B)) + D-1,
\]
whose subtraction leads to the claim.\qedhere
\end{proof}

\begin{remark}\label{rem:Reduction_dipole}
Notice that the previous lemma can be interpreted as
preorder preservation under dipole reduction: If $B$ is a
graph with more than 4 vertices and with a dipole $\delta \subset B$,
and if $\pi$ is a Wick contraction of $B$ that contains the vertices of $\delta$,
then $\check\pi= \pi|_{V(B^-)}$ is a well-defined Wick contraction of
$B^-= B\setminus \delta$.  By the same token of Lemma \ref{lem:unique_ExtContr},
the maximality of $\pi$ is equivalent to the
maximality of $\check\pi$.\end{remark}

\begin{proposition}[Maximal Wick contraction is unique]\label{propo:max_unique}
There exist a unique maximal Wick contraction of a melonic connected graph $B$
and it is the one constructed above, called $\pi_{\max}$.
\end{proposition}

\begin{proof}
Let $\pi$ be a Wick contraction of $B$ and suppose that $ \pi > \pi_{\max}$.
Since the two Wick contractions must differ, there is an index $k$
with $0< k\leq {p-1}$, such that
\begin{itemize}
\itemb in the notation of Definition \ref{def:pi_max},
the pair $(v_k,w_k)$, is not in $\pi$ and
\itemb for all indices $l$ with  $ k < l $, $(v_l,w_l) \in\pi_{\max} $.
\end{itemize}
This means that $(v_k,w_{j_0})\in \pi$ for some other
white-vertex index with $j_0\neq k$, and $(v_{l_0},w_k)\in \pi$ for some
black-vertex index $l_0 \neq k$.
By definition of $k$,
the assumption $ \pi > \pi_{\max}$ is not detected by the higher levels $B\hp l$
for $l>k$, for there all the Wick contractions coincide. By Remark \ref{rem:Reduction_dipole}, the preorder relation is preserved
at each level. Concretely, this means that if
$\pi^{\lceil i \rceil} := \pi|_{V(B\hp i)}$ and
$\pimax^{\lceil i \rceil}  := \pimax|_{V(B\hp i)}$, then
\[
\pi^{\lceil i \rceil}   > \pimax^{\lceil i \rceil}  \qquad \forall i \geq k,
\]
as consequence of $\pi>\pimax$. But then we have
(at least for $k=i$) a Wick contraction $\pi^{\lceil i \rceil} $ with two elements
$(v_k,w_{j_0})$ and $(v_{l_0},w_k)$ (exhibited above),
whereas $(v_k,w_k)$ forms a dipole, which is a contradiction to
Lemma \ref{lem:MaxWick}. Then $\pi>\pimax$ is false.\qedhere


\end{proof}
\begin{lemma}[Maximum of faces, I]\label{lem:MaxFaces_connected}
For any connected melonic $D$-coloured graph $B$,
\[
\# F(\pimax (B)) =
p (D-1)  +1 .
\]
\end{lemma}
\begin{proof}
We show that per insertion of each dipole, $D-1$ new faces appear.

Notice that in the construction of both a melon and its maximal Wick pairing
$\pimax$, which `remembers' the inserted dipole by pairing its two
vertices, the following happens. For
$i<p-1$ let $[\pimax(B)]^{i-1}$ be the graph $\pimax(B)$
excluding all dipole insertions after the $i-1$-th dipole, in the tree-construction of $B$ (the end graph does not depend of the way we order these
dipole insertions, but this level does: it is however easily seen that
the next fact holds for any way we enumerate them, as far as
the ascending order holds when one goes from parents to children). Then $[\pimax(B)]^{i+1}$ has the extra $D-1$  faces of the new dipole $\delta$; these $D-1$ correspond to the colours that differ from the colour (say $c$) of $\delta$. As for the colour-$c$ face in the dipole,
it is only the extension (by two edges, of colours $0$ and $c$),
of the length-$2$ face contained already in $[\pimax(B)]^{i-1}$, and hence
it is not new (it just has length $4$ instead of length $2$).
\par
Then $
\# F(\pimax (B)) $ is the faces of the 2-vertex melon,
plus $(D-1)$ faces per dipole insertion. But the number of dipole insertions is $p-1$ (starting with the 2-vertex melon), since we have to arrive to $2p$ total vertices.
\end{proof}

\begin{corollary}[Integral of a connected melon]\label{coro:uniwersalnosci_spójnych_melonów}
For each connected melonic graph $B$,
\[\lim_{N\to\infty} \frac 1N\int_{(\C^N)^{\otimes D}} B(T) \gauss =1. \]
\end{corollary}

\begin{proof}
Thanks to the uniqueness
in Lemma \ref{lem:unique_ExtContr}, we can extract from
the sum over Wick contractions of $B$ a single maximal element with corrections at lower orders of $N$, so
\[\frac 1N\int_{(\C^N)^{\otimes D}} B(T) \gauss &  = \frac1N
\sum_{\pi, \text{\tiny Wick} }
A [ \pi(B)]
\notag \\
\notag
& =  \frac1N \big \{
A [ \pimax(B)] +  \mathrm{o}(1)
\big \} \\
\notag
& =
\notag  \frac1N N^{\#F(\pimax(B)) - p(D-1)} [1+ \mathrm o(1)]\\
& = 1 + \mathrm  o(1). \notag
\]
The latter equation is thanks to Lemma \ref{lem:MaxFaces_connected} (the `small o' notation o$(g)$ summarises `contributions vanishing as $N\to\infty$' when divided by $g(N)$. In the case at hand o$(1)$,
$g(N) \equiv 1$).
\end{proof}

\subsection{Multi-trace integrals}\label{sec:Multitraces}

\begin{definition}[Cycle]\label{def:cycle}
For $k\geq 2$, a \textit{$k$-cycle} of a Wick contraction $\Pi$ of
connected $D$-coloured graphs $B_1, B_2,\ldots, B_k$
is a sequence of black $x_i \in V(B_i)$ and white vertices $y_i \in V(B_i)$,
such that $(x_i,y_{i+1}) \in \Pi $, for $i\in\Z_{k}$ (i.e. $i+k \equiv i$).
\end{definition}

\begin{definition}[Swap]\label{def:swap}
A \textit{swap} of a Wick contraction $\pi$
of a graph $B$, which need not be connected,
at two pairs $(x,y), (x',y')\in\pi$,
is a new Wick contraction $\pi'$ given by
\[\pi'= \big ( \pi \setminus \{ (x,y),(x',y')  \}\big )
\cup \{  (x,y'),(x',y)  \} . \]
\end{definition}

If $\pi_1$ and $\pi_2$ are Wick contractions of $B_1$ and $B_2$ respectively,
consider  a swap in
$(\pi_1\cup \pi_2)(B_1\dot\cup B_2)$ at pairs
$(x,y)$ and $(x',y')$. Then the
number of faces $\#F'$ of the new swapped graph is
\[
\#F' = \#F(\pi_1(B_1)) + \#F(\pi_2(B_2))- D. \label{deficit}
\]
The reason is that both
$(x,y)$ and $(x',y')$ were initially contained each in $D$
faces. After the swap, the two of $c$-coloured faces
are joined to a longer $c$-coloured face, for each $c=1,\ldots, D$.
This creates a deficit of $D$ faces, and \eqref{deficit} follows.
\par

We denote by $\Gr n(B_1,\ldots,B_n)$ the set of \textit{connected} $(D+1)$-coloured graphs of the form
$\Pi(\dot\cup_{i=1}^n B_i)$, where $\Pi$ is a Wick contraction of $\dot \cup_i B_i$  that satisfies
conditions (i) and (ii) given by:
\begin{itemize}
 \item[(i)] $\Pi$ is obtained by applying a finite sequence
 of swaps to $\cup_{i=1}^n \pimax\hp i$, being
 $\pimax\hp i$ is the maximal Wick pairing of $B_i$
 as constructed before Definition \ref{def:pi_max}.
As a part of condition (i), each one of such swaps should
change the number of connected components (equivalently, the swap cannot
take place at vertices of the same connected component).
%

 \item[(ii)] Consider the following substitutions:\begin{enumerate}
                          \item[(a)] For $i=1,\ldots, n$, replace the graph $B_i$ inside $\Pi(\dot\cup_{i=1}^n B_i)$ by a vertex
                          with $d_i$ incident (straight) half-edges, being
                          $d_i$ the number of
                          Wick contracted pairs of $B_i$ with
                          different $B_j$'s ($j\neq i$).

                          \item[(b)]  For each value of $k$, replace each $k$-cycle
                          of $\Pi(\dot\cup_{i=1}^n B_i)$ among $k$ different $B_{i_1},\ldots,B_{i_k}$ graphs ($\{i_1,i_2,\ldots,i_k\}
                          \subset \{1,2,\ldots, n\}$) by a $k$-valent
                          vertex adjacent to $k$ wiggly half-edges.

                          \item[(c)] Join the half-edges that arise
                          from (a) with those that arise from (b) if they were
                          Wick contracted in the original graph.

                          \end{enumerate}
            One thus obtains a graph (`thin graph') with half-edges complemented by
            half-edges of a different type (straight or wiggly). Then condition (ii)             reads: (a), (b) and (c) yield a tree.
\end{itemize}

The set $\Gr n(B_1,\ldots, B_n)$ is interesting due to the
tension among the two points above:
Since the graph must be connected and since we start
from $\cup_i \pimax^i$, which yields a disconnected
$\cup_i \pimax^i(B_i)$, a non empty sequence of swaps
is required to yield connectedness. In turn, these swaps
are restricted by
not leading to loops when the graphs are `seen from far away';
concretely, we mean (ii).

\begin{example}\label{ex:Gr}
We display in the second row the graph that (ii) associates
to each coloured graph in the first row:\footnote{These coloured graphs are evidently bipartite, but we do not draw the different
colours of the vertices. In the second row, angles
are different for sake of visualisation, but mathematically not relevant.}
\begin{subequations}%
\[%
\runter{\includegraphics[width=1.46cm]{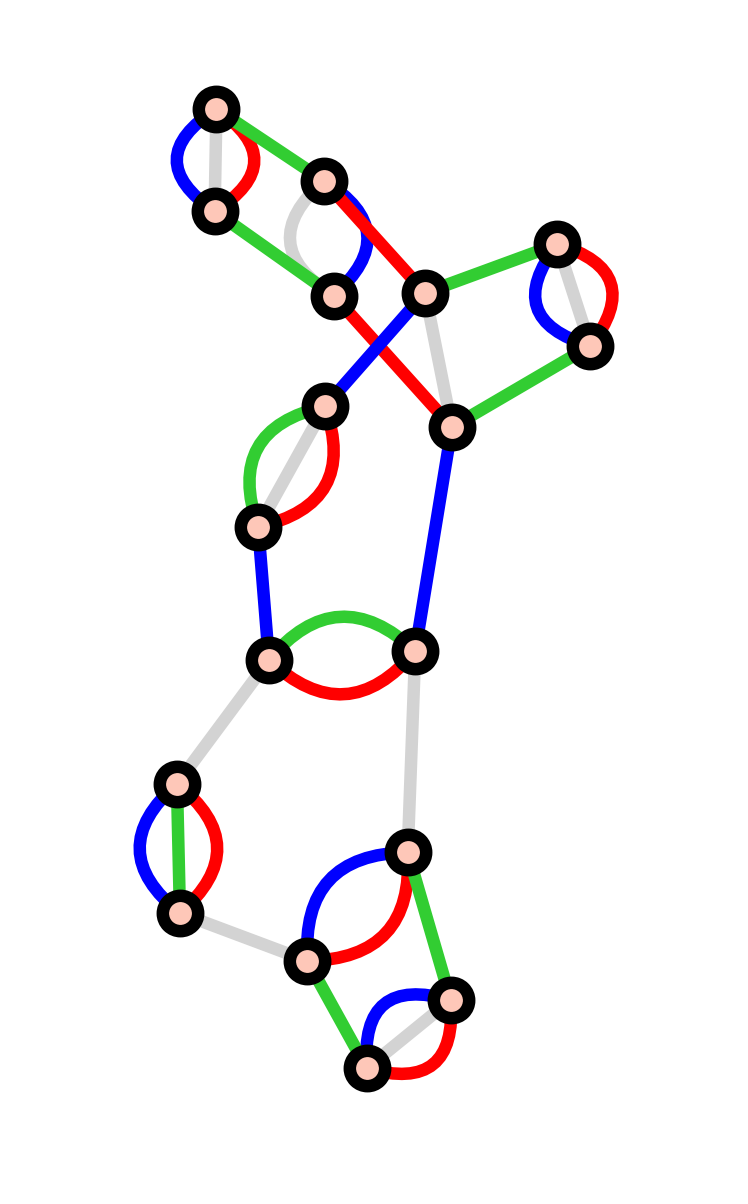}} & &
\runter{\includegraphics[width=1.96cm]{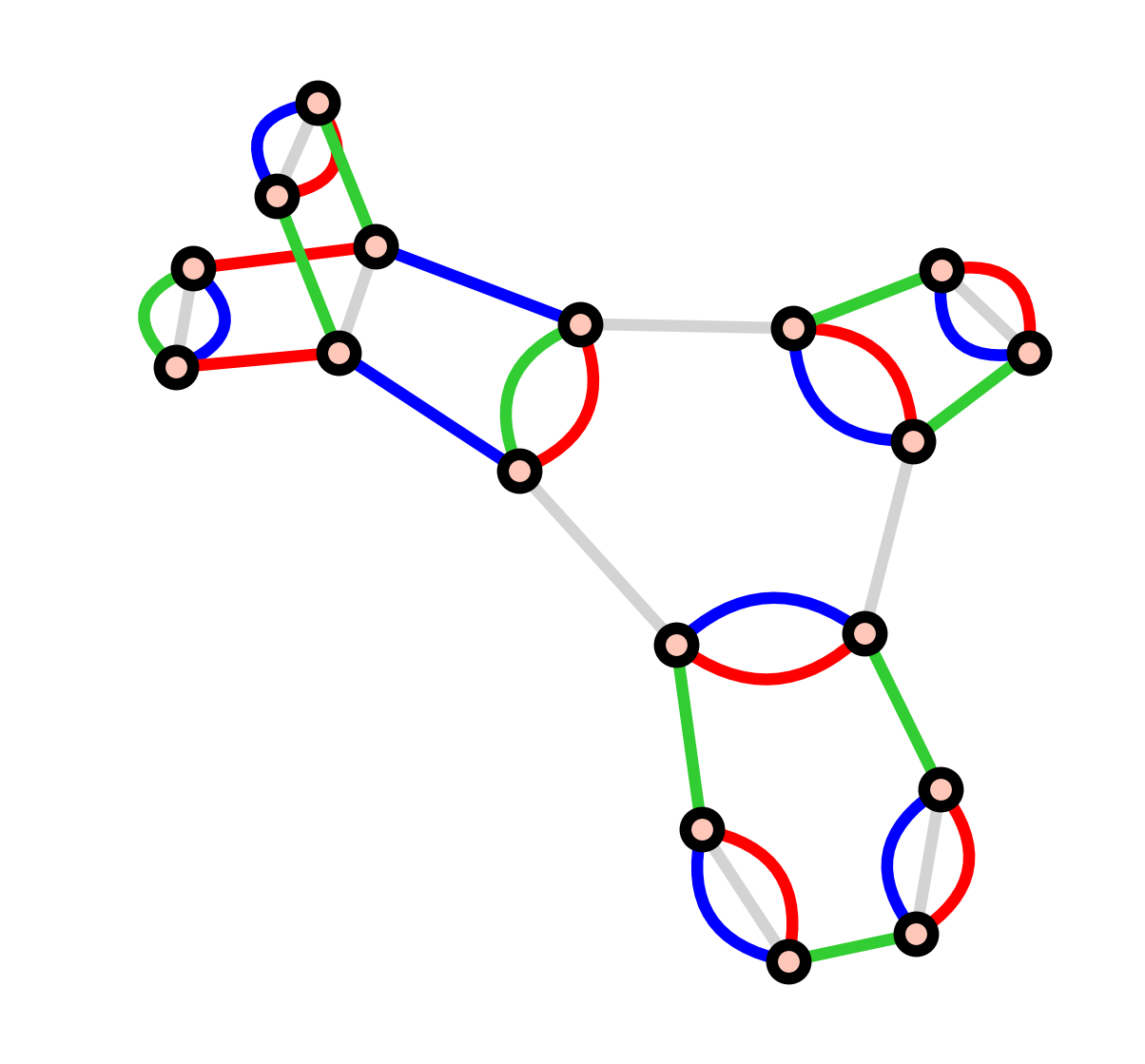}} &&  \runter{\includegraphics[width=2.28cm]{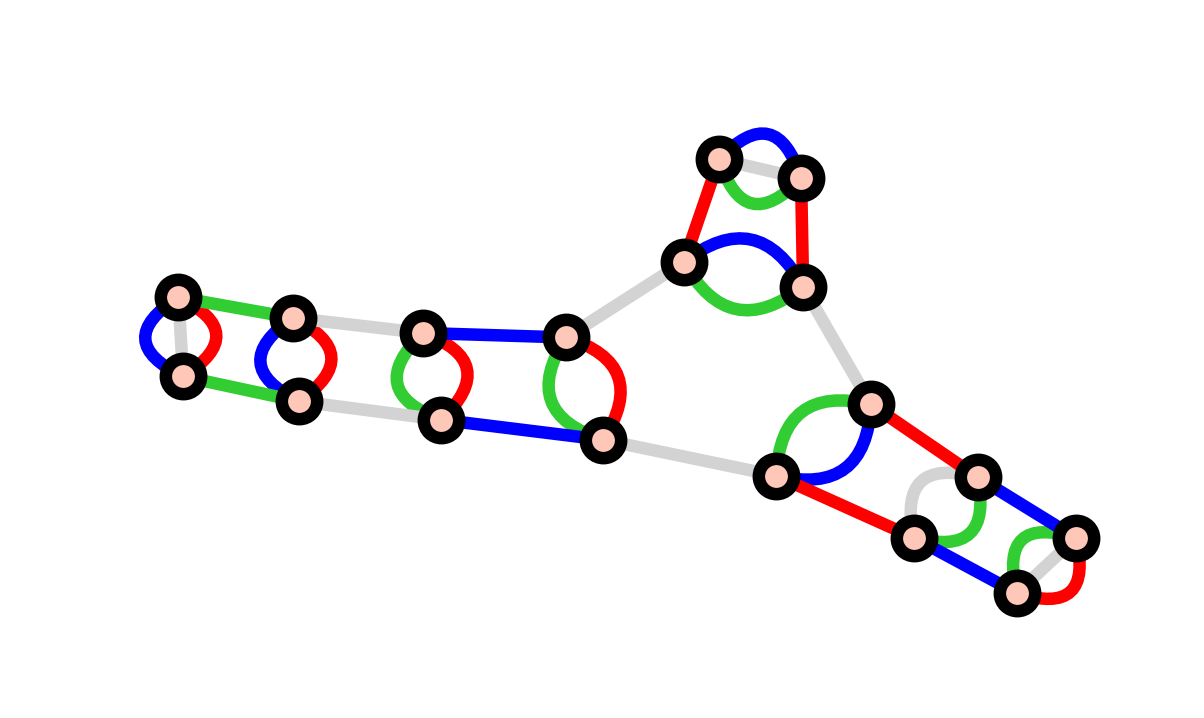}} &&  \runter{\includegraphics[width=1.89cm]{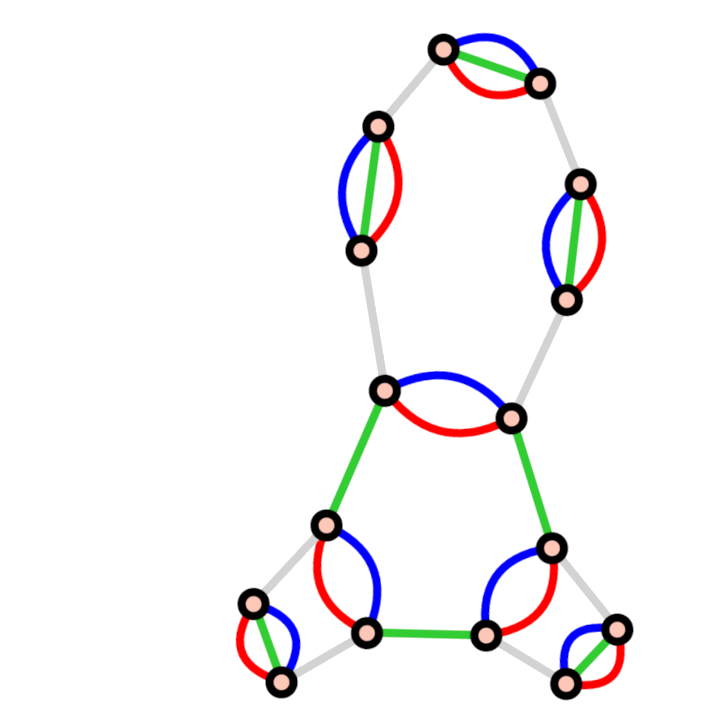}}\\
\runter{\includegraphics[width=1.38cm]{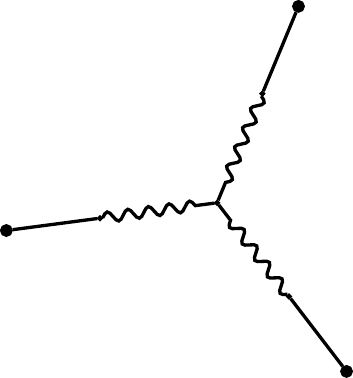}} &
 &\runter{\includegraphics[width=1.36cm]{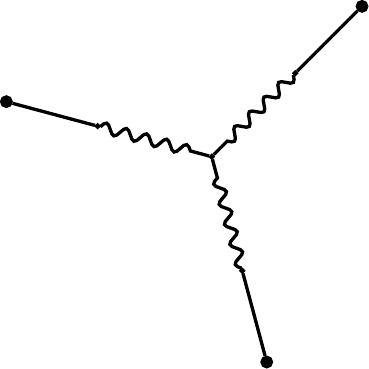}} & &\runter{\includegraphics[width=1.6cm]{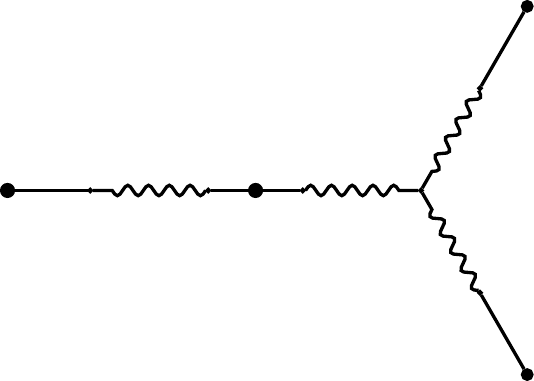}}\quad
  &&\runter{\includegraphics[width=1.36cm]{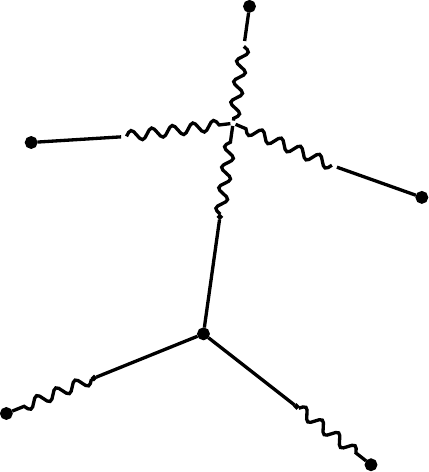}}
\]\end{subequations}%
\end{example}

Before proving the main proposition of the section, which
reveals what $\Gr n$ is good for, it is convenient to determine
the precise maximal number of
faces.
\begin{lemma}[Maximum of faces, I\hspace{1.15pt}I]\label{lem:MaxNumFaces}
For $B_1,\ldots, B_n$ melonic, connected, $D$-coloured graphs,
the maximal number of faces of $\Pi(\dot\cup _i B_i)$
for any maximal Wick contraction $\Pi$ that obeys Condition (ii) above is
\[
\# F[ \Pi(\dot\cup _i B_i) ] =  (D-1) \times ( P - n + 1 ) +1
\]
where $2P$ is the number of vertices of $\dot\cup _i B_i$.
\end{lemma}
\begin{proof}
By induction in $n$.
The case $n=1$ is Lemma \ref{lem:MaxFaces_connected}, so we proceed
to assume the induction hypothesis for $n-1$.\\
Notice that by assumption, since $\Pi(\dot\cup _i B_i)$
is a tree after performing (ii.a, ii.b, ii.c) above,
one can disconnect $\Pi(\dot\cup _i B_i)$
by a swap, to obtain a graph of two connected components. More specifically,
there exist an index $m_0\in\{1,2,\ldots,n\}$, such that the aforementioned
swap maps $\Pi(\dot\cup _i B_i)$ to the graph $\Pi_1(\dot\cup _{i\neq m_0} B_i) \dot\cup \Pi_0(B_{m_0})$,
where $\Pi_{0}$ and $\Pi_{1}$ are Wick contractions of their respective arguments.
Due to the relation \eqref{deficit}, the swap
$\Pi(\dot\cup _i B_i) \mapsto \Pi_1(\dot\cup _{i\neq m_0} B_i) \dot\cup \Pi_0(B_{m_0})$
\textit{increases} the face number in $D$,
since it is the reversed swap to that in Eq. \eqref{deficit}, increasing the number of connected components.
Then
\[
\#F [\Pi(\dot\cup _i B_i) ] & =\#F [\Pi_1(\dot\cup _{i\neq m_0} B_i)]  +\#F [\Pi_0(B_{m_0})] -  D \notag \\
\notag & =  \big\{ (D-1) \times ( P -p_0 - n + 2 ) +1  \big\} +\#F [\Pi_0(B_{m_0})] - D
\]
by induction hypothesis,  being $2p_0 = \#V(B_{m_0})$ therein.
By assumption, $\Pi$ is maximal, hence so must be $\Pi_0$. But then $\Pi_0$  must be $\pimax^{m_0}$
by uniqueness, as in Lemma \ref{lem:unique_ExtContr}. Here $\pimax^{m_0}$ is constructed by
Algorithm \ref{algorithm1}, whose number of faces is known to us thanks to
Lemma \ref{lem:MaxFaces_connected}. This yields
\[ \notag
\#F [\Pi(\dot\cup _i B_i) ] = \big\{(D-1) \times ( P -p_0 - n + 2 ) +1 \big\} + [p_0 (D-1) +1] -D\]
as claimed.
\end{proof}

\begin{proposition}\label{prop:Mn=Gn} For $B_1,\ldots, B_n$,  connected $D$-coloured melonic graphs, let
\[ \mathscr M_n(B_1,\ldots, B_n)
:= \{ \Pi(\dot\cup_{i=1}^n B_i) \hspace{1pt}\text{connected} \hspace{1pt}: \Pi \text{ is maximal Wick contraction }\!\}. \notag
\]
Then \vspace{-2ex}
\[\notag
\mathscr M_n(B_1,\ldots, B_n) =  \Gr n(B_1,\ldots, B_n).
\]

\end{proposition}

\begin{proof}
We have to prove both set contentions,
$\Mr n \subset \Gr n$ and $\Mr n \supset \Gr n$ (we have simplified the notation,
leaving out the $B_j$'s, which are all fixed).

\begin{itemize}

\itemb $\Mr n \supset \Gr n $. Suppose that $\Pi$ is a Wick contraction of $\dot\cup_{i=1}^n B_i$
such that
 $G:=\Pi(\dot\cup_{i=1}^n B_i) \in \Gr n $. In particular, due to (i)
 $G$ is obtained from $\dot\cup_i \pimax^i B_i$ applying certain number of swaps,
 after which $G$ is connected.
 Due to (ii) reduces to a tree after applying (a), (b), (c),
 which means that one had to apply exactly $n-1$ swaps to arrive to from
  $\dot\cup_i \pimax^i B_i$  to $G$. Hence, from the starting number of faces
  \[\notag
  \# F (\dot\cup_i \pimax^i B_i) = \sum_{i=1}^n p_i  [ (D-1)  + 1] = (D-1) P +n
  \]
  (with $P=\sum_{i=1,\ldots,n} p_i$), we perform  $n-1$ swaps,
  and observe that the number of faces of $G$ must be
 \[\notag
   \# F (G) = (D-1) P+n -D(n-1) = (D-1) \times (P-n+1) +1
 \]
 after writing $D=(D-1)+1$.
 Then, by Lemma \ref{lem:MaxNumFaces}, $\Pi$ is maximal, hence $G=\Pi(\dot\cup_{i=1}^n B_i) \in \Mr n$.
\vspace{.5ex}
\itemb $\Mr n \subset \Gr n $. We prove this by contradiction.
Pick $G:=\Pi(\dot\cup_{i=1}^n B_i) \in \Mr n$ and assume that $G$ is not in $\Gr n$.
Since $G\in \Mr n$, $G$ is connected, and $G\notin \Gr n$ only can mean
that $G$ does not satisfy (i) or does not satisfy (ii), the two
defining properties of $\Gr n$.
\begin{itemize}
\item[-] Assume that $G$ does not satisfy (i).
Since $G \in \Mr n$ is connected, it must be
obtained by swaps of $n$ individual Wick contractions $\pi^i$,
one for each $B_i$, $i=1,\ldots,n$.
By assumption there is at least one subgraph $B_{j_0}$, for some $j_0\in\{1,\ldots,n\}$, for which $\pi^{j_0}$ is not $\pimax^{j_0}$ --- else,
$G$ must had been generated from swaps from all maximal $\pimax^1,\ldots,\pimax^n$.
The above means that $\pi^{j_0}< \pimax^{j_0}$ strictly. Since performing (a, b, c) of (ii)
one obtains a tree from $G$, the number of performed swaps
to arrive from $\{\pi^i(B_i)\}_{i=1,\ldots,n}$ to $G$ is $n-1$.
This allows us to compute
\[ \notag
\#F ( G ) &= \sum_{i=1}^n
\#F (\pi^i(B_i))
- (n-1) D \\\notag
& <  \#F (\pimax^i(B_i))
- (n-1) D \\ &\notag
=  \sum_{i=1}^n [ p_i(D-1) +1 ] - (n-1) D\notag
\]
and conclude that $G=\Pi(\dot\cup_{i=1}^n B_i)$, initially picked from $\Mr n$,
cannot be maximal. Contradiction. Hence (i) must hold.

\item[-] Now assume that $G$ does not satisfy (ii).
This means that (a,b,c) above led to a thin graph
that is not a tree. The thin graph must then have an edge we can cut
without disconnecting the graph (thereby undoing the corresponding swap
in $G$ that corresponds to that edge in the thin graph). If this is still not a tree, we can repeat the process until one arrives to a tree; call $l$ the number
of cuts we needed. In the process,
the number of faces suffered from an increment in $l \times D$, one $D$
per each swap undone. Call $\Gamma$ the coloured graph that corresponds to the tree
after undoing the $l$ swaps. Then
\[
\#F (\Gamma) =
\#F (G)  + l\times D.
\]
But then $\Gamma$ is connected, it is a Wick contraction of $\dot \cup_i B_i$, and $\Gamma > G$, so $G$ cannot be maximal, contradicting that $G\in \Mr n$. So (ii) must hold.\qedhere
\end{itemize}
\end{itemize}

\end{proof}

\begin{corollary}\label{coro:independenceD}
For fixed connected, melonic $D$-coloured graphs $B_1,\ldots, B_n$, the number
$\#\mathscr M_n(B_1,\ldots, B_n)$
does not depend on $D$ and it depends on the $B_i$'s only
through their number of vertices.
\end{corollary}
\begin{proof}
It follows from Proposition \ref{prop:Mn=Gn} by observing that
$\# \Gr n(B_1,\ldots, B_n)$ does not depend on $D$. The key point
is that each element in
$\Gr n(B_1,\ldots, B_n)$
is obtained by swaps of \textit{unique} maximal Wick contractions
of the components $B_i$. Indeed, even though
different melonic graphs---even if they have the
same number of vertices---will generally lead to different maximal
Wick contractions, the crux of the matter is that for each $B_i$ being melonic and connected, $\pimax^i$ is uniquely given by Prop. \ref{propo:max_unique}. Having a single maximal
Wick contractions per $B_i$, which is only a bijection
of the white vertices with the black vertices of each $B_i$,
the number of ways to arrange these in a tree structure by means
of a sequence of swaps is determined only by $(p_1,\ldots, p_n)$,
being $2p_i =\#V(B_i)$.\qedhere
\end{proof}

\section{The universal melonic measure and melonic polynomials}\label{sec:MelPolyn}

In the next subsection we introduce the universal measure (or rather, a family of universal measures)
in the context of tensor rank universality, and in Section \ref{sec:melnonic_polyn_sol}
the solutions to the integrals.

\subsection{The universal melonic measure}\label{sec:univ_melon_meas}

\begin{definition}
Let $\mathsf{u}_p$ be the $3$-coloured graph\footnote{These graphs were named  `cyclic melonic' by S. Carrozza, it seems to me. Here
these are our universal graphs.}
that corresponds to $p-1$ dipole insertions always in
colour-$1$ leaves in the tree-construction of a 3-coloured graph, so
\[\mathsf{u}_2 =\vone ,\quad
\mathsf{u}_3 = \Qone ,\quad \mathsf{u}_4 = \Aone , \quad
\mathsf{u}_5 =\tenQ\vphantom{a}^{\!\text{\tiny 1}}, \ldots \]
The \textit{universal melonic measure}  is the measure on $(\C^N)^{\otimes 3}$
that reads
\[
\exp\Big\{-N^{2} \sum_{p=2}^{\infty} t_p \frac{\mathsf{u}_p(T)}{p} \Big\}
\gauss \qquad \with t_2,t_3,\ldots, \in \R.
\]
The infinite sum will truncate, and should be understood as follows:
for any melonic tensor measure there is a choice of the finitely many
$p$'s such that $t_p\neq 0$, such that the measures
are indistinguishable at large-$N$.
The exact values of the $t_p$'s, and which ones do not vanish,
are determined later.
\end{definition}
\noindent
The cumulant of a melonic invariant $C$ with respect to the
universal measure is given by
\[\mathbf{E}\hp {\mathrm{\tiny c, }N}_{\mathbf t}
[\tfrac 1N C] :=
\frac 1N\int\conn_{(\C^{N})^{\otimes 3}}
\mathsf{u}_{ q}(T) \exp\bigg\{-N^{2} \sum_{p=2}^{\infty} \frac{t_p}{p}\mathsf{u}_p(T) \bigg\} \gauss, \qquad 2q=\#V(C),\notag
\]
($t_p=0$ for almost all $p$ as before).
The restriction `conn.' retains only connected
graphs, see the end of Section \ref{sec:integrals_intro}.
%

\begin{theorem}[Twofold universality of tensor integrals]\label{thm:uniwersalnosc_całek}
For $D\in \Z_{\geq 3}$, and $B_1,B_2,\ldots,B_n$
$D$-coloured graphs with $2p_i =\#V(B_i)$, let
\[
\Delta_{N,D}(\{B_i\}_i):=
N^{(D-1)n -D}
 \int\conn_{(\C^{N})^{\otimes D}}
 \prod_{i=1}^n B_i(T)
\gauss
-
N^{2n -3}
 \int\conn_{(\C^{N})^{\otimes 3}}
 \prod_{i=1}^n
\mathsf{u}_{ p_i}(T)
\gauss. \notag
\]
If  $B_1,B_2,\ldots,B_n$ are all melonic,
this difference vanishes at large-$N$,
\[
 \lim_{N\to \infty } \Delta_{N,D}(B_1,\ldots,B_n) =0.
\label{inv_universality}
\]
In particular,
for any $\tilde D\in\Z_{\geq 3}$, if
$\tilde B_1,\tilde B_2,\ldots,\tilde B_n$
are $\tilde D$-coloured graphs with $\#V(B_i)=2p_i$ for all $i$, then
\[
\lim_{N\to \infty } \Bigg\{&
N^{(D-1)n -D}
 \int\conn_{(\C^{N})^{\otimes D}}
  B_1(T)
 \cdots B_n(T)
\gauss \notag
 \\-&
N^{(\tilde D-1)n -\tilde D}
 \int\conn_{(\C^{N})^{\otimes \tilde D}}
  \tilde B_1(T)
\cdots \tilde B_n(T)\gauss\Bigg\} = 0\,. \label{limDtildeD}
\]
\end{theorem}
Although the previous is one universality theorem, the previous equations emphasise two different aspects:
first, `universality of melonic tensor invariants' as stressed by
Eq. \eqref{inv_universality}, and secondly `rank universality' or
Eq. \ref{limDtildeD}. Notice also that testing the melonicity of a graph is not
an $D$-independent procedure. We do use $D$ to
produce the boolean data $B_i \mapsto \{0,1\}$
telling whether $B_i$ is melonic, but once this fact (which is
an assumption of the theorem above) is known, $D$ disappears.

\begin{proof}
For sake of notation let  $\mathsf B = \dot\cup_{i=1}^n B_i$,
and let $2P= \#V(\mathsf B)=2p_1+\ldots +2p_n$. So
\[ \notag
 \int\conn_{(\C^{N})^{\otimes D}}
 \prod_{i=1}^n B_i(T)
 &\stackrel{\text{\tiny Eq. }\ref{integral_def}}{=}
 \sum_{\Wick\, \pi}
 A[\pi(\mathsf B )] \\ &\notag
\stackrel{\text{\tiny Eq. }\ref{amplitude}}{=}
 \sum_{\Wick\, \pi}
 N^{(D-1) \#F[\pi(\mathsf B)]-
 \tfrac{D-1}2 \#V[\pi(\mathsf B)]\notag
 }    \\\notag
 &\stackrel{\text{\tiny Lem. }\ref{lem:MaxNumFaces}}{=} \Bigg\{
 \sum_{\pi\in \Mr{n} (\mathsf B)}
 N^{(D-1)( P-n+1 )+D - P(D-1) } \Bigg\} \times
\big(1+ \mathrm o(1) \big) \\\notag
& \stackrel{\text{\tiny Prop. }\ref{prop:Mn=Gn}}{=}
\# \Gr n (B_1,\ldots, B_n) \times  N^{-(D-1)n+1} \times
\big(1+ \mathrm o(1) \big)
\]
In deriving this set of equalities, we applied two definitions and
split the integral in Wick contractions that maximise faces, and the rest,
which leads to $\mathrm o(1)$ terms.
On the other hand, the same arguments, one by one, will lead to
\[ \notag
 \int\conn_{(\C^{N})^{\otimes 3}}
 \prod_{i=1}^n u_{p_i}(T)
 =
\# \Gr n (u_{p_1},\ldots,u_{p_n}) \times  N^{-2n+3} \times
\big(1+ \mathrm o(1) \big)
 \]
  Now we observe that
$\# \Gr n (B_1,\ldots,B_n) =\#
\Gr n (u_{p_1},\ldots,u_{p_n}) $ by Corollary \ref{coro:independenceD},
so inserting  the result of these two integrals
(with their respective factors, which render both terms finite and
$N$-independent) in $\Delta_{N,D}(\{B_i\}_i)$,
one gets that $\Delta_{N,D}(\{B_i\}_i)$ is $
\mathrm o(1)$, and the limit follows. Then
Eq. \ref{limDtildeD} is obtained from
\begin{equation}
\lim_{N\to \infty } \big[\Delta_{N,D}(\{B_i\}_i)-\Delta_{N,\tilde D}(\{\tilde B_i\}_i) \big]=0. \qedhere
\end{equation}
\end{proof}

\subsection{Solving integrals by means of melonic polynomials}\label{sec:melnonic_polyn_sol}
\noindent
To the best of our knowledge, the next polynomials are new, so we chose a name for them.
\begin{definition}
For $p_1,p_2,\ldots,p_n\in \Z_{\geq 1}$, the $n$-th \emph{melonic polynomial}
is given by
\[ \label{meloniczny_wielomian_n}
\mathrm{Mel}_n (p_1,p_2,\ldots, p_n) :=
p_1p_2\cdots p_n  \times (p_1+p_2+\ldots+p_n -1)_{n-2}
\]
being $(x)_m$ the $m$-th Pochhammer symbol or $x!/(x-m)!$.
\end{definition}

Depending on the numbers $p_1,\ldots,p_n$, the statement presented next is a theorem or a conjecture.
Let us represent a sequence of integers $
p_1 \geq p_2 \geq \ldots\geq p_n$ by (not labelled) Young tableaux
with $p_1$ vertical boxes next to $p_2$ vertical boxes,
etc, from left to right, and refer to $(p_1,\ldots, p_n)$ as the shape of the diagram.
 These diagrams play only an organizational role
only to decide what has been proven and what is conjectured, and else do not influence the content
of the statement.
\begin{theoremconjecture}\label{theoconj}
The melonic integrals (i.e. for $B_1,\ldots, B_n$ melonic, $D$-coloured graphs)
\[
I_{N,D,n}(B_1,\ldots,B_n) :=
N^{n(D-1) -D} \int_{(\C^{N})^{\otimes D}} B_1(T)B_2(T)\cdots B_n(T) \gauss
\]
satisfy
\[
\lim_{N\to \infty} I_{N,D,n}(B_1,\ldots,B_n)
=\mathrm{Mel}_n (p_1,p_2,\ldots, p_n).
\]
In particular, the right hand side does not depend on $D$, and it
depends on the graphs only through their number of vertices, but
not on their combinatorial details.  This statement is a theorem
for all diagrams of shape $(p_1,p_2,p_3,\ldots,p_n)$
that appear coloured in Figure \ref{fig:progress}
being $p_i=\#V(B_i)/2$, and else a conjecture (since in the integral
one can permute the $B_i$'s in descending order of the $p_i$'s,
we do not care about the order).
\end{theoremconjecture}

The reason why the shape-(7,1) diagram in  Figure \ref{fig:progress}
is not coloured is because we prefer to invest time in the
proof for all partitions $p_i$ (the
generation of $14$-vertex graphs is relatively easy,
but excluding repeated graphs needs to test
automorphisms to compare O$(10^3)$ graphs; also their integration takes some time).

Although the integral does not motivate a proof by induction in $n$,
the description of the melons via the melonic polynomial seems to facilitate this.
Nevertheless, we will not prove this theorem in that way. We shall
use another strategy and report this elsewhere \cite{Melonic_polyns}. For the time being, we content ourselves
with a case-by-case proof, leaving the unaddressed cases as conjectured.
It is worth noticing that we will not use the statement \ref{theoconj} in this article (if the conjecture is false,
nothing in the sequel depends on it).

\begin{proof}[Proof of the theorem part.] We proceed by cases.\\

\noindent
\textbf{Case arbitrary $D\in\Z_{\geq 3}$ and $n= 1$}.
These are the pink diagrams in Figure \ref{fig:progress}, but for all
(also undepicted) $n=1$ diagrams of shape $(p)=$ one pile of $p$ boxes, for any $p$, the claim holds
thanks to Corollary \ref{coro:uniwersalnosci_spójnych_melonów}.\\
\begin{figure}
\begin{minipage}{.45\textwidth}
\includegraphics[width=1.099\textwidth]{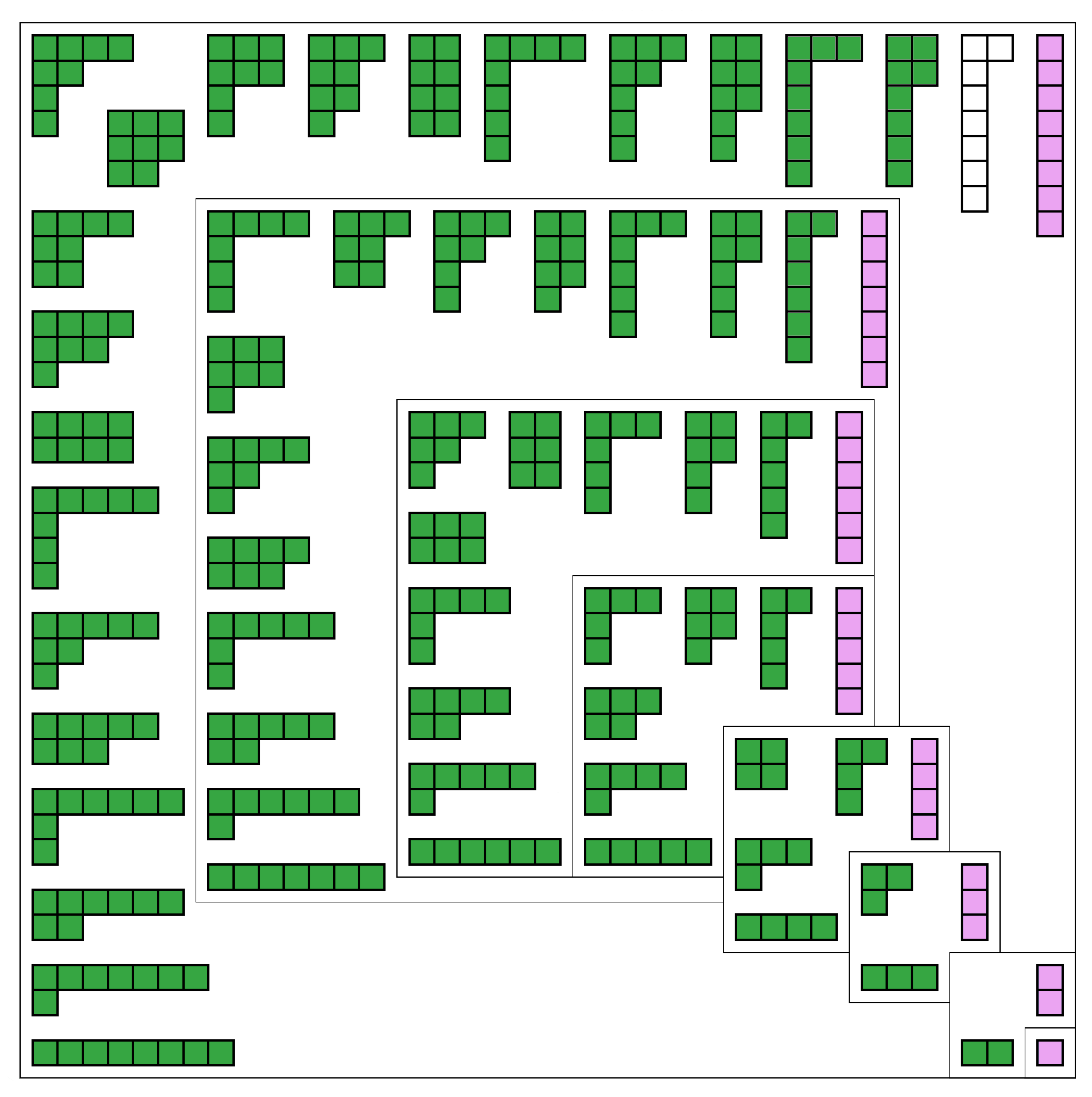}
\end{minipage}\!\!\!
\begin{minipage}{.53\textwidth}
\caption{(Left). Each of these tables of width $n$ is associated with integers
$p_1\geq p_2,\ldots \geq p_n$, where $p_i$ is the depth of
the $i$-th column (from left to right; e.g. to the diagram
\eqref{5_3}, $p_1=5, p_2=3$; for the single white diagram,
$p_1=7,p_2=1$). Green means `proved by explicit computation',
an empty or an absent diagram  means `conjectured' and pink means `proved by Corollary \ref{coro:uniwersalnosci_spójnych_melonów}'.
For thousands of integrals of
20 and 22 vertices, all computed integrals agree with the
statement of \ref{theoconj}, but the next rows of explored
integer partitions become emptier
and emptier. [Credit for the picture:
 Wikipedia user R. A. Nonenmacher, modified (coloured)
 by the author.]
\label{fig:progress}}
\end{minipage}
\end{figure}
\begin{figure}
\qquad\qquad\!\!\includegraphics[width=.4506\textwidth]{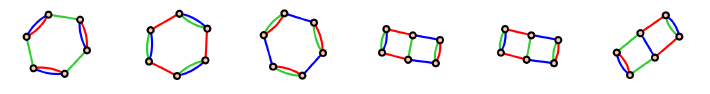}\newline
\footnotesize (a) \\[2ex]
\qquad\qquad\includegraphics[width=.536\textwidth]{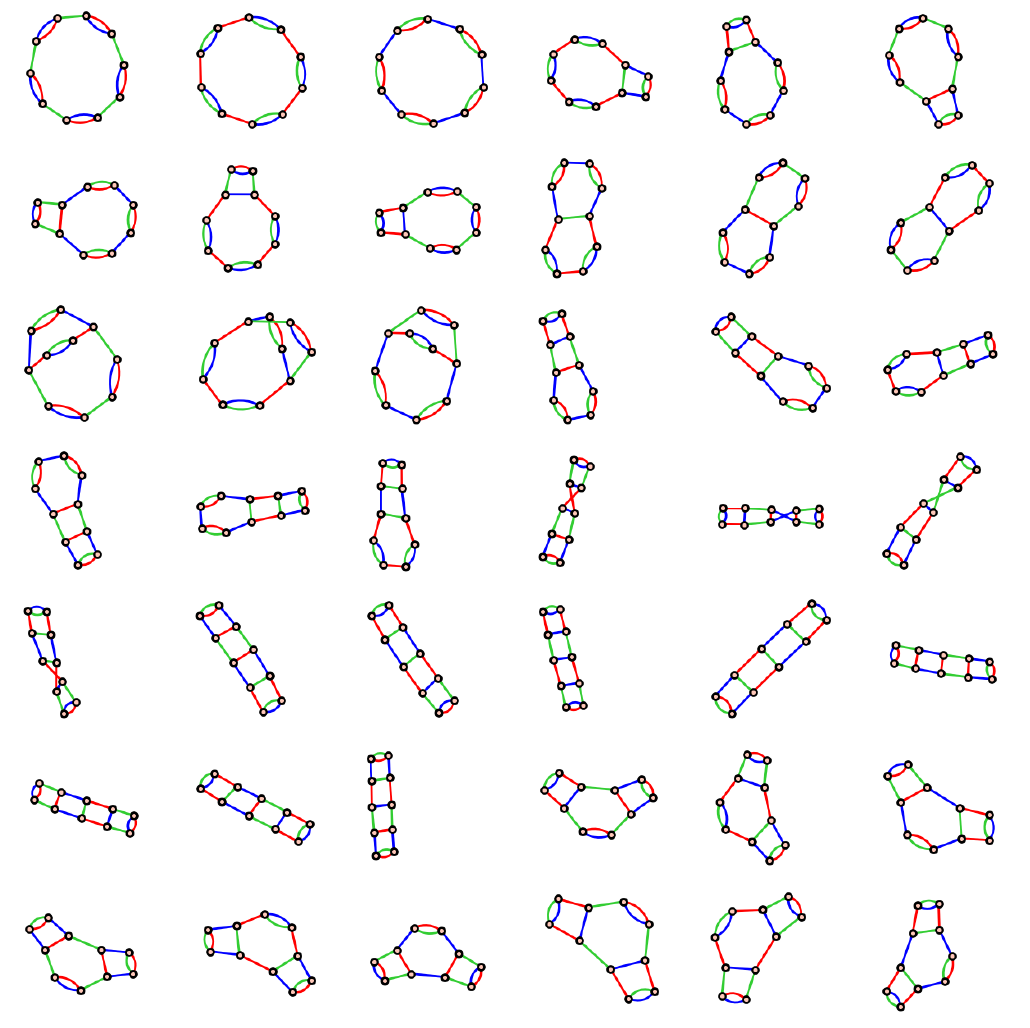}\newline
\footnotesize (b)
 \caption{Some melons for $D=3$: (a) 6-vertices melons and
 (b)  10-vertices melons\label{fig:10pt_melons}. The diagram
 \ref{5_3} requires having computed integral of their products (not all of them are independent).}
\end{figure}

\noindent
\textbf{Case $D=3$ and $n\geq 2$}. If a Young diagram of form ($p_1,p_2,\ldots,p_n$) with $
p_1 \geq p_2 \geq \ldots\geq p_n$ is green in
Figure \ref{fig:progress}, we mean that the statement is true
after computing the integral for
\textit{all} melons $B_1$ of $2p_1$ vertices, \ldots,
and \textit{all} melons $B_n$  of $2p_n$ vertices, etc.
For instance, \vspace{-3ex}
\[\runter{\includegraphics[width=3ex]{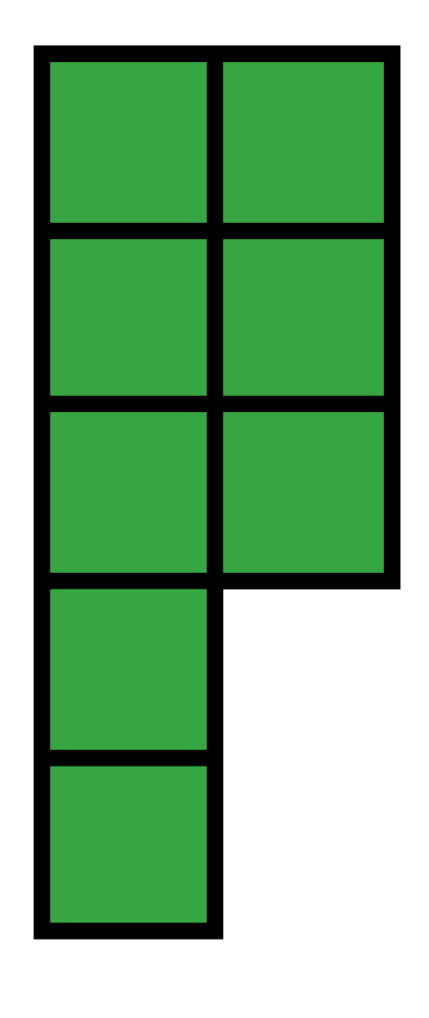}} \label{5_3}
\]
means that we have proven that the statement holds for
all inputs from all $10$-vertex melons $B_1$
and all $B_2$ (see Fig. \ref{fig:10pt_melons}). This means
that the diagram \ref{5_3} being coloured green implies that
that we know the integrals implied there and their leading order
agrees with the value of the melonic polynomial $\mathrm{Mel}_2(5,3)$.
For empty diagrams this is a conjecture and for those non-depicted too (for which
$\sum_{i=1,\ldots,n} p_i\geq 20$).\\

\noindent
\textbf{Case arbitrary $D\in\Z_{\geq 4}$ and $n\geq 2$}.
By Theorem \ref{thm:uniwersalnosc_całek} these integrals are
also universal in $D$, in particular due to Eq. \eqref{limDtildeD}.
Since the statement held for $D=3$, so does for $D\in \Z_{\geq 4}$.\qedhere
\end{proof}

\section{Equivalence of interactive theories at large-$N$}
\label{sec:Conclusion}

\noindent
Let us prove the utility of the universal melonic measure.
\begin{theorem}\label{thm:interactive_theory_equivalence_D}
For arbitrary $D,\tilde D\in \Z_{\geq 3}$ let
\[
C,B_1,B_2,\ldots,B_m \qquad &\text{melonic, connected $D$-coloured}\notag \\ \notag
\tilde C,\tilde B_1,\tilde B_2,\ldots,\tilde B_m \qquad &\text{melonic, connected $\tilde D$-coloured}
\]
be such that
\[
\#V(C)=
\#V(\tilde C) \,\,\und\,\,
\#V(B_i)=
\#V(\tilde B_i):= p_i, \,\,\forall i=1,\ldots,m.
\]
Consider the measures
\[\notag
\exp\big[-N^{D-1}\sum_i  \frac{g_i}{p_i} B_i  \big]\gauss \text{ on $(\C^N)^{\otimes D}$} \\
\exp\big[-N^{\tilde D-1}\sum_i  \frac{g_i}{p_i} \tilde B_i  \big]\gauss
\text{ on $(\C^N)^{\otimes \tilde D}$}\notag
\]
If one of the cumulants $\E\hp {\mathrm{\tiny c, }N}_{\mathbf g} [C]$
and $\tilde\E\hp {\mathrm{\tiny c, }N}_{\mathbf g} [\tilde C]$
exists, the so does the other, and both agree at large-$N$
\[\notag
\lim_{N\to\infty} \frac1N \tilde \E\hp {\mathrm{\tiny c, }N}_{\mathbf g} [C] =
\lim_{N\to\infty} \frac1N \E\hp {\mathrm{\tiny c, }N}_{\mathbf g} [\tilde C]\,.
\]
Further, they both exist when the following universal melonic measure cumulant does:
\[\lim_{N\to\infty} \frac1N  \mathbf E\hp
{\mathrm{\tiny c, }N}_{\mathbf t} [\mathsf{u}_p],\qquad 2p=\#V(C)
,\qquad   t_q = \sum_{\substack{k\,\, \text{\tiny such that}
\\ \#V(B_k)=q}} g_k \label{tis} \]
and in that case, the three previous cumulants agree,
provided they are evaluated, as above, at the same $\mathbf g$ value, of course, and at $t_i$ given by Eq. \eqref{tis}.
\end{theorem}

\begin{proof}
Let $2 M= \max \{ \#V(B_i) : i=1,\ldots,m\}$, in other words $M$
is the maximum of $\{p_i\}_i$.
We start by directly computing with the measure without tilde,
rewriting
\[\mathsf S_{\mathbf g}(T) &:=
\sum_i g_i \frac{B_i(T)}{p_i} \notag
 = \sum_{p=2}^M
 \bigg[ \sum_{k : \#V(B_k)=2p } B_k(T) \bigg] \\
 & =
  \sum_{p=2}^M \mathscr s_{\mathbf g} (p) \qquad \with\qquad\mathscr s_{\mathbf g} (p) := \frac 1p
 \sum_{ \substack{k \,\,\text{\tiny such that} \\ \#V(B_k)=2p} } g_k B_k(T)
 = \frac 1p \sum_{i=1}^{\ell_p}
 \gamma_{p,i} b_{p,k}(T)
 .  \label{smallS}
\]
where $b_{p,i}$ is
the $i$-th graph in $\{B_1,\ldots, B_m\}$
with $2p$ vertices (the order does not matter, as we will sum over $i$)
and $\gamma_{p,i}$ is its respective coupling. Also $\ell_p$
is the number of graphs in $\{B_1,\ldots, B_m\}$
that have $2p$ vertices.
On the other hand,\[\notag
\frac1N   \E\hp {\mathrm{\tiny c, }N}_{\mathbf g} [C]
&=\int\conn_{(\C^{N})^{\otimes D}}
\frac 1N C(T) \ee^{-N^{D-1} \mathsf S_{\mathbf g}(T) } \gauss
\\
&=\int\conn_{(\C^{N})^{\otimes D}}
\frac 1N C(T) \frac{(-1)^n}{n!}  \sum_{n=0}^\infty
N^{n(D-1)}
\mathsf S^n_{\mathbf g}(T)  \gauss \notag.
\]
By assumption the cumulant converges at $\mathbf g$, so we can swap the integral and the sum
\[
\frac1N  \E\hp {\mathrm{\tiny c, }N}_{\mathbf g} [C]
&= \sum_{n=0}^\infty
{(-1)^n}
N^{n(D-1)-1}
\int\conn_{(\C^{N})^{\otimes D}}\frac{1}{n!} C(T)
\mathsf S^n_{\mathbf g}(T)  \gauss.\]
To expand this, let us use Eq. \eqref{smallS}.
The integral (ignoring everything outside it) reads
\[\notag
I_n&:=\sum_{ j_2+ \ldots +j_M=   n }
\int\conn_{(\C^{N})^{\otimes D}}
\frac{1}{j_2!j_3!\cdots j_M! } C(T)
\mathscr s_{\mathbf g}^{j_2} (2)
\mathscr s_{\mathbf g}^{j_3} (3)
\cdots\notag
\mathscr s_{\mathbf g}^{j_M}(M) \\
& =
\sum_{ j_2+ \ldots +j_M=   n }
\int\conn_{(\C^{N})^{\otimes D}}
\frac{1}{j_2!j_3!\cdots j_M! } C(T)
 \Big[ \frac{1}{2} \sum_{i=1}^{\ell_2} \gamma_{2,i}  b_{2,i}(T)\Big] ^{j_2}
\cdots
 \Big[ \frac{1}{M} \sum_{k=1}^{\ell_M} \gamma_{M,k}  b_{M,k}(T) \Big] ^{j_M} \notag
\]
By universality of
melonic tensor integrals (Theorem \ref{thm:uniwersalnosc_całek})
one can replace \textit{inside the integral} each $b_{p,i}$ by $\mathsf u_{p}$,
up to $\mathrm o(1)$ terms.
Due to this, and
observing that
the restriction in the sum implies
that integration of an invariant with  $n+1$ connected
components is to be performed (the extra component being $\mathsf u_p$
associated to $C$), one obtains for
$
[1+ \mathrm o(1)] \cdot I_n$ the expression
\[ &\phantom{=}
\sum_{ j_2+ \ldots +j_M=   n }
\frac{N^{2(n+1)-3}}{ N^{(D-1)(n+1)-D} }
\int\conn_{(\C^{N})^{\otimes 3}}
\frac{ \mathsf u_p(T)  }{j_2!j_3!\cdots j_M! }
 \Big[ \sum_{i=1}^{\ell_2} \gamma_{2,i}\frac{ \mathsf u_{2}(T)}{2} \Big] ^{j_2}
\cdots
 \Big[  \sum_{k=1}^{\ell_M} \gamma_{M,k}\frac{ \mathsf u_{M}(T)}{M}\Big] ^{j_M} \notag \\
 &=\sum_{ j_2+ \ldots +j_M=   n }
\frac{N^{2n-1}}{ N^{(D-1)n-1} }
 \int\conn_{(\C^{N})^{\otimes 3}}
\frac{  \mathsf u_p(T) }{j_2!j_3!\cdots j_M! }
 \Big[ t_2 \frac{  \mathsf u_{2}(T)}{2} \Big] ^{j_2}
\cdots
 \Big[t_M \frac{  \mathsf u_{M}(T)}{M} \Big] ^{j_M}   \notag
\]
by the definition of the $\gamma$-couplings below Eq.. \eqref{smallS}
and by definition of the $t_i$ couplings. Inserting $I_n$
into the initial cumulant notice that the $N$ powers characteristic of rank $D$
cancel out and those of characteristic of rank $3$ appear,
\[
\lim_{N\to \infty}
\frac1N   \E\hp {\mathrm{\tiny c, }N}_{\mathbf g} [C]
& = \notag
\lim_{N\to \infty}
\sum_{n=0}^\infty
{(-1)^n}
N^{2 n-1}
\int\conn_{(\C^{N})^{\otimes D}}\frac{1}{n!} \mathsf u_p(T)
 \Big[ \sum_{p=1}^M t_p\frac{\mathsf u_p(T)}{p}
 \Big]^n   \gauss \\
& = \lim_{N\to \infty}
\frac1N   \mathbf E\hp {\mathrm{\tiny c, }N}_{\mathbf t} [\mathsf u_p] \label{last}
 \]
It is obvious that, replacing the notation
$D\to \tilde D$, $g\to \tilde g$, $C\to\tilde C$, $B_i\to \tilde B_i$,
all arguments hold, as we have assumed that $B_i$ and $\tilde B_i$
have the same number of vertices. This is obvious from
Theorem \ref{thm:uniwersalnosc_całek} too. Then we also obtain
$\lim_{N\to \infty}
\frac1N   \tilde \E\hp {\mathrm{\tiny c, }N}_{\mathbf g} [\tilde C]=
 \lim_{N\to \infty}
\frac1N   \mathbf E\hp {\mathrm{\tiny c, }N}_{\mathbf t} [\mathsf u_p]$,
and the three cumulants agree.
\\

\noindent
Now observe that the chain of arguments that took us
from Eq. \eqref{smallS} to Eq. \eqref{last} can be reversed (we always have an `if and only if').
Hence, if
 $
\frac1N   \mathbf E\hp {\mathrm{\tiny c, }N}_{\mathbf t} [\mathsf u_p]$
exists, so do
$\frac1N   \tilde \E\hp {\mathrm{\tiny c, }N}_{\mathbf g} [\tilde C]$
and $\frac1N   \E\hp {\mathrm{\tiny c, }N}_{\mathbf g} [C]$ and their
large-$N$ limits agree.
\end{proof}

%


\subsection{Applications}\label{sec:applications} It is well-known
that some solutions of tensor models explicitly
depend on $D$, and the reader might wonder how to
conciliate this with the universality we proved above.
We exemplify this situation with a quartic tensor model.
\par

Call $V_i\hp D$ the four point graph which consists of
an insertion of a $i$-dipole in the two-point melon (also
called `pillow'), thus \[ V_1\hp 3 = \vuno,
\qquad V_2\hp 3 =\vdos,\qquad
V_3\hp 3 =\vtres\]
with higher-colour number analogues
for tensors of $D$ indices.
Consider the measure
\[ \exp\Big\{-N^{D-1} \frac 12\sum_{i=1}^M
g_i V_i\hp D (T)
\Big\} \gauss\, \qquad g_1,\ldots,g_M \in \R \label{measure_corollary} \]
on  $(\C^N)^{\otimes D}$, with $M\leq D$. The correlation functions
for this measure are known under the restrictions:  $g_1=g_2=\ldots = g_M$
with $D=M$.  We apply  Theorem  \ref{thm:interactive_theory_equivalence_D} to obtain the two-point function for $D\neq M$ without the
$g_i$ parameters having to be equal.

\begin{corollary}[Of the Dartois-Eynard-Nguyen solution and Thm. \ref{thm:interactive_theory_equivalence_D}] \label{coro:quartic_model}
The measure \eqref{measure_corollary}
has a critical locus given by $\sum_{i=1}^M g_i+1/4=0$ for any $D\geq 3$ and for any $M=1,2,\ldots,D$. In fact, the large-$N$ two-point function reads
\[
m_2(g_1,\ldots,g_M)
=
\frac{-1+\sqrt{1+4g_1+\ldots+4g_M} }{g_1+\ldots+g_M}
\,. \label{solution_corollary}
\]
\end{corollary}

The critical locus (for $M=3$) is depicted in Figure \ref{fig:critlocus} (a)
and its solution (for $M=2$) in Figure \ref{fig:critlocus} (b).
The interesting aspect of this corollary are however not the plots, but
that this fact holds for any $D$ and $M$.


\begin{figure}
\[
\includegraphics[width=.37\textwidth]{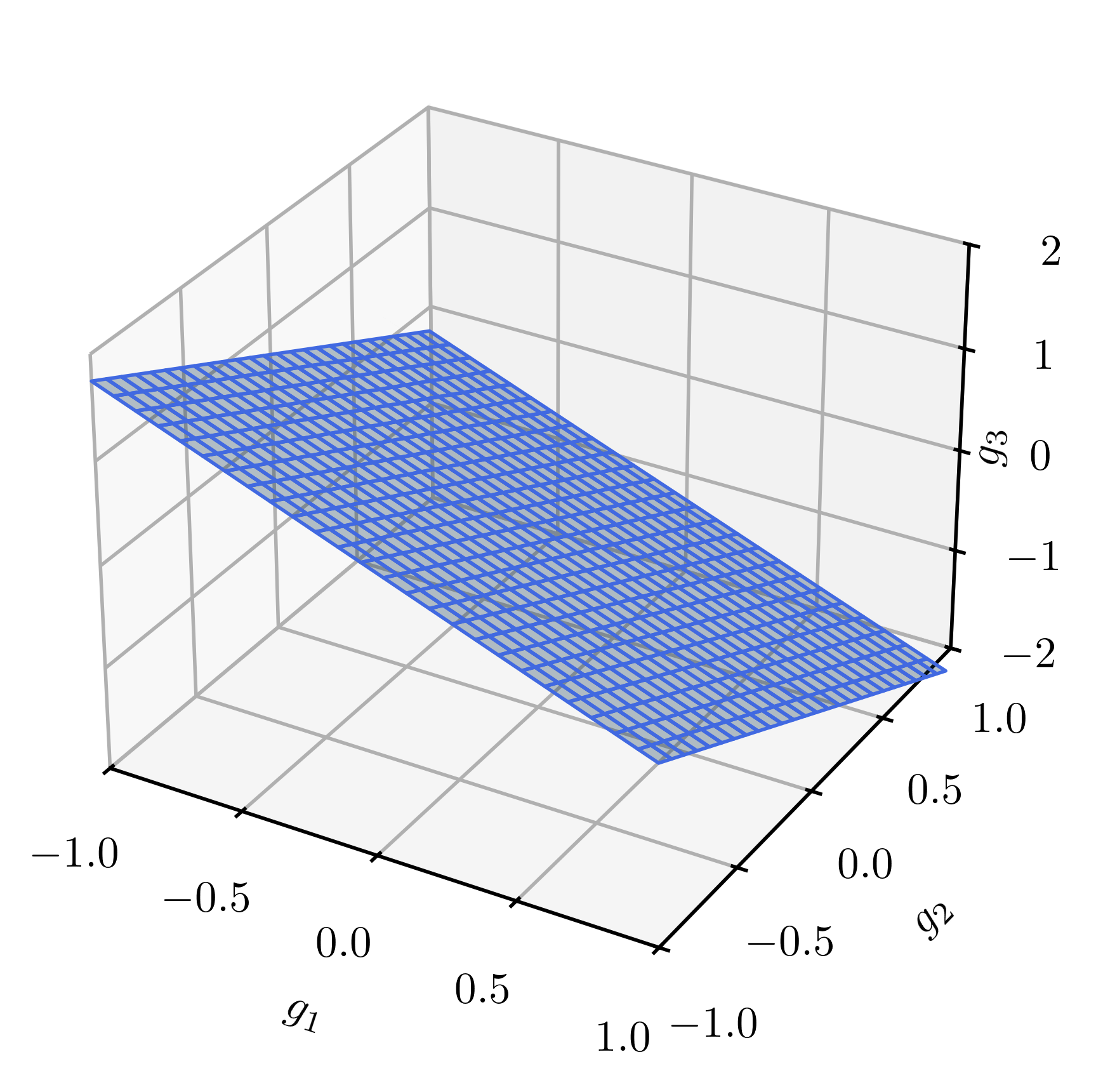} && \notag
\includegraphics[width=.39\textwidth]{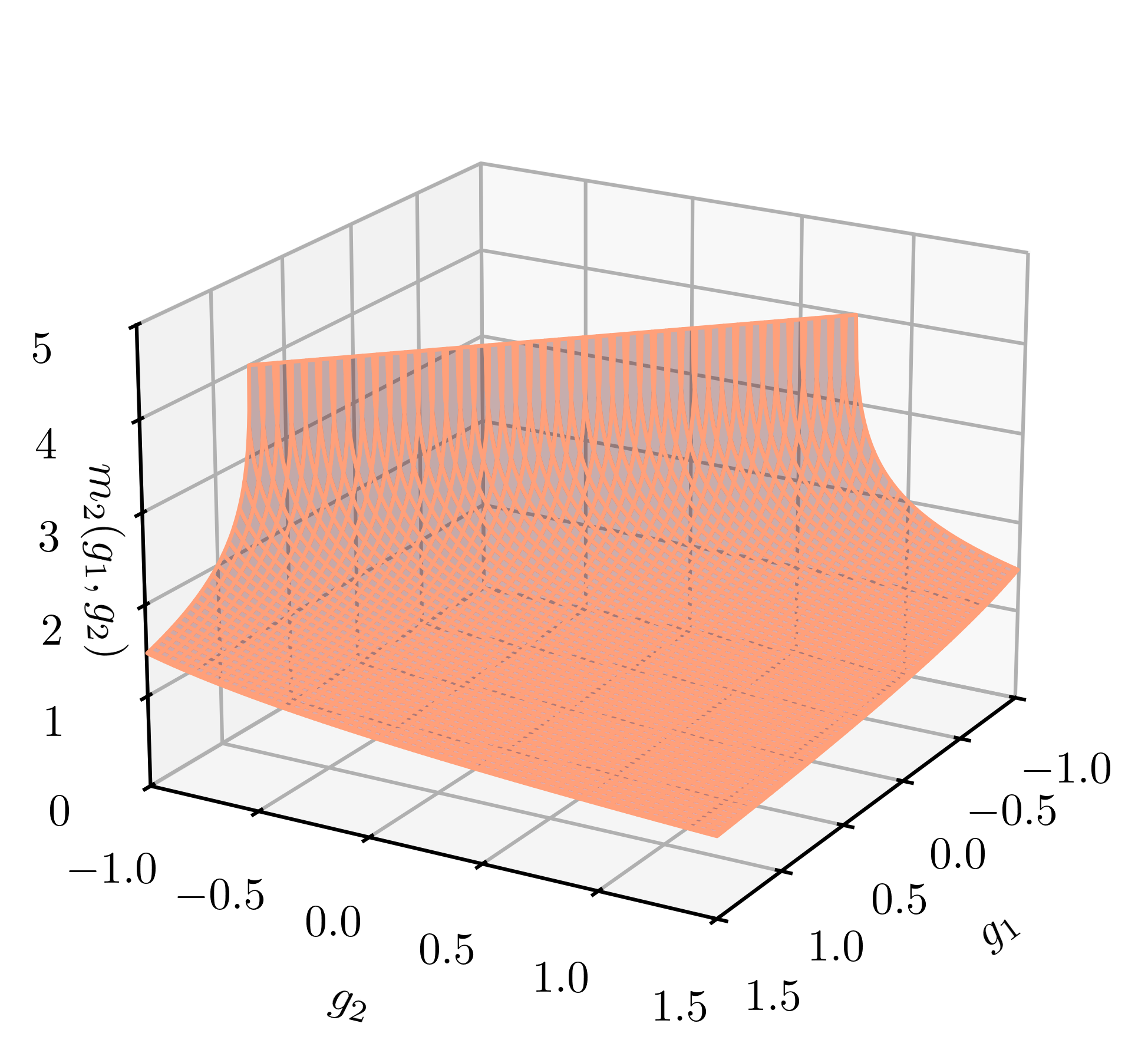} \\ \text{(a)}\qquad\qquad & &\hspace{-1ex}\text{(b)}
\phantom{abcsdfasdf}\notag \]
\caption{(a) Critical locus $4g_1+4g_2+4g_3+1=0$ of the measure \eqref{measure_corollary}
($M=3$, $D\geq 3$), as follows from the Universality Theorem \ref{thm:interactive_theory_equivalence_D}, using the
Dartois-Eynard-Nguyen solution.
(b) Exact solution $m_2(g_1,g_2)$
of the large-$N$ two-point function in the case $M=2$. Observe we plot two aspects with different parameters, and are not meant to be compared; instead, observe that the projection of the truncated straight line that shows divergence in (b) to the $(g_1,g_2)$-plane of (b) is the critical locus $4g_1+4g_2+1=0$ of $M=2$. \label{fig:critlocus}}
\end{figure}
\begin{proof}[Proof of the corollary.]
As auxiliary, let us use the model
with measure $\dif \nu_{D,g}(T) : =\exp[-N^{D-1} \frac g2\sum_{i=1}^D
V_i\hp D (T)
] \gauss$ the large-$N$ two-point function given by
$\mathfrak m_{2,D}(g) = \tfrac{1}{gD}(-1+\sqrt{1+4gD })$,
which has been solved in \cite{eynard_dartois_nguyen}.
By the Universality Theorems \ref{thm:uniwersalnosc_całek} and \ref{thm:interactive_theory_equivalence_D},
the universal melonic measure corresponding to $\dif \nu_{D,g}(T)$
is the measure
 $\exp[-N^{2} \frac {t_2}2
\vuno
] \gauss$ on $(\C^N)^{\otimes 3}$
with $t_2=Dg$. But the Dartois-Eynard-Nguyen \cite{eynard_dartois_nguyen} solution itself,
and Theorem
 \ref{thm:interactive_theory_equivalence_D}
 imply that
 the two-point function $ \mathfrak m^{\text{\tiny univ}}(t_2)$ of the universal measure
 (at $N\to \infty$) is
\[ \mathfrak m^{\text{\tiny univ}}(t_2)= \mathfrak m_{2,3}(g)|_{g=t_2/3} =
\tfrac{1}{t_2}(-1+\sqrt{1+4t_2 })\]
(this is $t_2=3g$ in Thm  \ref{thm:interactive_theory_equivalence_D}).
Coming back to our original measure \eqref{measure_corollary},
 Thm  \ref{thm:interactive_theory_equivalence_D} says  again that
 its two-point function at large-$N$ must be $
\mathfrak m^{\text{\tiny univ}}(t_2)|_{t_2=g_1+\ldots+g_m}$, which
is precisely the $m_2(g_1,\ldots,g_M)$ in Eq. \eqref{solution_corollary},
featuring the claimed critical locus.
\end{proof}




\section{Conclusions}

\subsection{Summary of the results}
We have proven a new type of universality of melonic random tensors that manifests in a twofold way
in the limit $N\to \infty$:
\begin{itemize}
 \itemb \textit{Tensor rank universality}
 or the independence of the number of indices of a tensor
while performing an integral (more precisely, Thm. \ref{thm:uniwersalnosc_całek})
and the independence of melonic tensor models of the number of indices  (Thm. \ref{thm:interactive_theory_equivalence_D}).

\itemb \textit{Universality
of the invariants} (`traces'),
as far as these are melonic and share the same number of vertices
(Thm. \ref{thm:interactive_theory_equivalence_D}).
\end{itemize}
None of these follows from the famous Gaussian Universality \cite{gurau2014universality}. Gaussian universality
is a deep result that relates melonic $2p$-order correlators with
the $p$-th power of the two-point function, but it does not yield
 the cumulants of one model in terms of other model, and Theorem \ref{thm:interactive_theory_equivalence_D} does (whence we
 claim that our universality is new).
We also introduced
\textit{melonic polynomials} in Eq. \eqref{meloniczny_wielomian_n}, which solve
melonic integrals at large-$N$ and agree with  explicit
computations with \texttt{feyntensor} \cite{feyntensor}.

\subsection{Physics outlook}\label{sec:phys_outlook} The present result stresses the importance in the future of:
\begin{itemize}
 \itemb The
  study of non-melonic tensor models (examples are \cite{octahedra,LionniJohannes}).
 \itemb Non-melonic cumulants in melonic tensor models. (In  \cite{TensorBootstrapProgram} we proposed positivity bootstraps that include
non-melonic observables.)
\itemb Finite-$N$ methods, which surely detect $D$ (\texttt{feyntensor} \cite{feyntensor} confirms this, but also positivity bootstraps \cite{toriumi2026} at finite-$N$ do so for the $\vuno$-model).
\itemb  It is pertinent to comment on a family of 2-matrix models
\cite{ABAB_MonteCarlo} with a potential that includes the terms $\tfrac14
\Tr[ h (A^4+B^4) + 2q g ABAB+ 2(1-q) g ABBA]$ (with $0\leq q\leq 1, g,h\in\R$).  This is so constructed, that the sum of the couplings
of the two latter words is always $g$.  Monte Carlo simulations by the author show that,  despite this restriction,
the phase portraits in the $(g,h)$-plane are different for different $q$. From that view point, the implication of Theorem \ref{thm:interactive_theory_equivalence_D} for the tensor case that, at large-$N$,
there exist \textit{one} quartic, unitary invariant melonic tensor model---and that the effective coupling in the universal measure is
the sum of the couplings of all quartic operators---seems strong, specially contrasting with the previous matrix case, in which
criticality is not the same for all convex combinations ($q=1/2$ \textit{seems} critical for fixed $h,g$).
\itemb Universality is not expected in other theories
\textit{based} on tensors like \cite{3Dbeta,us,fullward}, since these break unitary invariance, but it is worth exploring to which point it does.
\end{itemize}
\begin{figure}
\begin{center}
\begin{tikzpicture}[yscale=1.0,xscale=1.8]
  \node[fill=white] (z)
     at (1.3, 8.0) { };
  \node[fill=lightgray!50, font={\footnotesize},draw=gray] (a)
     at (4.8, 8.0) {Lemma \ref{lem:MaxWick}};
  \node[fill=lightgray!50, font={\footnotesize},draw=gray]  (b)
     at (4.8, 7.2) {Lemma \ref{lem:unique_ExtContr}};
  \node[fill=lightgray!50, font={\footnotesize},draw=gray]  (c)
     at (4.8, 6.4) {Proposition \ref{propo:max_unique}};
  \node[fill=lightgray!50, font={\footnotesize},draw=gray]  (d)
     at (4.8, 5.6) {Lemma \ref{lem:MaxFaces_connected}};
  \node[fill=lightgray!50, font={\footnotesize},draw=gray] (e)
     at (4.8, 4.8) {Corollary \ref{coro:uniwersalnosci_spójnych_melonów}};
  \node[fill=lightgray!50, font={\footnotesize},draw=gray]  (f)
     at (4.8, 4.0) {Lemma \ref{lem:MaxNumFaces}};
  \node[fill=lightgray!50, font={\footnotesize},draw=gray]  (g)
     at (4.8, 3.2) {Proposition \ref{prop:Mn=Gn}};
  \node[fill=lightgray!50, font={\footnotesize},draw=gray]  (h)
     at (4.8, 2.4) {Theorem \ref{thm:uniwersalnosc_całek}};
  \node[fill=lightgray!50, font={\footnotesize},draw=gray]  (i)
     at (4.8, 1.6) {Theorem \ref{thm:interactive_theory_equivalence_D}};
  \node[fill=lightgray!50, font={\footnotesize},draw=gray]   (j)
     at (2.62, 3.2) {\cite{eynard_dartois_nguyen}};
  \node[fill=lightgray!50, font={\footnotesize},draw=gray] (k)
     at (2.62, 1.6) {Corollary \ref{coro:quartic_model}};
  \node[fill=lightgray!50, font={\footnotesize},draw=gray] (l)
     at (7.2, 4) {\texttt{feyntensor} \cite{feyntensor}};
  \node[fill=lightgray!50, font={\footnotesize},draw=gray] (m)
     at (7.2, 2.4) {Theorem-Conjecture \ref{theoconj}};
    \path[-stealth,thick ]
    (a) edge (b)
    (b) edge (c)
    (c) edge (d)
    (d) edge (e)
    (e) edge (f)
    (f) edge (g)
    (g) edge (h)
    (h) edge (i)
    (j) edge (k)
    (l) edge (m)
    (h) edge (m)
    (i) edge (k)
    ;
\end{tikzpicture}\notag 
\end{center}
\caption{Logic tree of this article, with
implications denoted by ordinary arrows.
\label{fig:logic_tree}}
\end{figure}
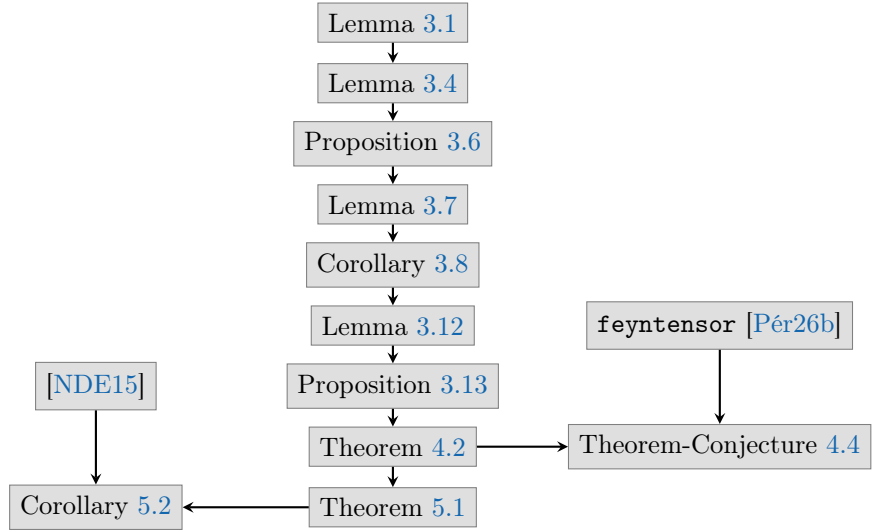
\subsection{Combinatorics outlook.}
For a large number of cases (Fig. \ref{fig:progress}) the Statement
\ref{theoconj} is a theorem, in others a conjecture whose
proof is work in progress \cite{Melonic_polyns}.
It is essential to remark that the Conjecture \ref{theoconj}
was not used here, and is a `leaf' in the `logic tree' in Figure \ref{fig:logic_tree} that
depicts the structure of this article: that is, if the conjecture
turns out to be wrong, the rest of the article remains intact. The orthogonal-invariant tensor ensemble \cite{On}---the cousin of the present unitary-ensemble---is the next step to explore universality.


\section*{Acknowledgements}
I thank R\u azvan Gur\u au  and Luca Lionni for
their  motivating questions and influential comments during a seminar
talk. SageMath\footnote{It was very motivating
to just press `\texttt{Perf}+\texttt{[Tab]}' on SageMath and seeing
\texttt{PerfectMatching} appear immediately after some comments by R. Gur\u au on his \cite{gurau2025nonfactor}. Both his comments and SageMath were the most influential
input to write \texttt{feyntensor}.} \cite{sagemath}
has been very helpful regarding the Statement \ref{theoconj}.
%

\bibliographystyle{alpha}

\end{document}